\pdfoutput=1
\documentclass[12pt]{amsart}

\usepackage{color}
\usepackage{graphicx}
\usepackage[mathcal, mathscr]{euscript}
\usepackage{bbm}
\usepackage{float}
\usepackage{amsmath,amsthm,amssymb,amsfonts,amscd,epsfig,latexsym,graphicx,textcomp}
\usepackage{tikz-cd}
\usepackage{psfrag}
\usepackage[all]{xy}
\xyoption{pdf}
\usepackage{stackengine}

\usepackage{comment}
\usepackage{hhline}
\usepackage{diagbox}
\usepackage{float}
\usepackage{caption}
\usepackage{eufrak}
\usepackage{slashed}

\usepackage{tikz}

\restylefloat{figure}

\usepackage{thmtools}
\usepackage{thm-restate}

\usepackage{hyperref}

\usepackage{cleveref}

\newtheorem{Theorem}{Theorem}

\newtheorem{Conjecture}[Theorem]{Conjecture}

\newtheorem{theorem}{Theorem}[section]

\newtheorem{proposition}[theorem]{Proposition}

\theoremstyle{definition}
\newtheorem{definition}[theorem]{Definition} 
\newtheorem{example}[theorem]{Example} 

\theoremstyle{remark}
\newtheorem{remark}[theorem]{Remark}
\numberwithin{equation}{section}
\newcommand{\ot}{\otimes}
\newcommand{\ra}{\rightarrow}

\newcommand{\BC}{\mathbb{C}}
\newcommand{\BR}{\mathbb{R}}

\def\del{\partial}

\def\del{\partial}

\usepackage{xpatch}
\makeatletter   
\xpatchcmd{\@tocline}
{\hfil\hbox to\@pnumwidth{\@tocpagenum{#7}}\par}
{\ifnum#1<0\hfill\else\dotfill\fi\hbox to\@pnumwidth{\@tocpagenum{#7}}\par}
{}{}
\makeatother    

\makeatletter
 \def\l@subsection{\@tocline{2}{0pt}{4pc}{6pc}{}}
\def\l@subsubsection{\@tocline{3}{0pt}{8pc}{8pc}{}}
 \makeatother

\begin{document}

\title{A new way to prove  configuration reducibility using gauge theory}

\thanks{}

\author{Scott Baldridge}
\address{Department of Mathematics, Louisiana State University\\
Baton Rouge, LA}
\email{baldridge@math.lsu.edu}

\author{Ben McCarty}
\address{Department of Mathematical Sciences, University of Memphis\\
Memphis, TN}
\email{ben.mccarty@memphis.edu}

\date{}

 \dedicatory{Dedicated to Sir Roger Penrose and Louis Kauffman for encouraging us to work on the four color theorem.}

\begin{abstract} We show how ideas coming out of gauge theory can be used to prove  configurations in the list of ``633 unavoidable configurations" are reducible. In this paper, we prove the smallest nontrivial example, the Birkhoff diamond, is reducible using our filtered $3$- and $4$-color homology. This is a new proof of a 111-year-old result that is a direct consequence of a special (2+1)-dimensional topological quantum field theory.  As part of the proof, we introduce the idea of  a state-reducible configuration. Because state-reducibility does not involve Kempe switches, this leads to an independent way to verify the proof of the four color theorem. We conjecture that  these gauge theoretic ideas  could also lead to a non-computer-based proof of it.
\end{abstract}

\maketitle

\section{Introduction}

In the past 10 years there have been four attempts to prove the four color theorem using gauge theory. These techniques were initiated by Kronheimer and Mrowka in their seminal papers on instanton homology of webs \cite{KM1, KM2, KM3}, which further inspired Khovanov and Robert's combinatorial foam evaluations \cite{KR}.  A student of the first author, Amit Kumar, has also recently produced a topological field theory with defects approach to the problem \cite{Kumar}.  The authors of this paper introduced {\em bigraded} and {\em filtered $n$-color homology} in \cite{BMcC}, which has its origins in $2$-factor homology in a paper by the first author \cite{BaldCohomology}.  Each of these four approaches are powerful in their own right and each emerged from studying different aspects of gauge theory.  

Until now, the best that these approaches could do was state an equivalence, i.e., ``The four color theorem is true if and only if  Theorem X is true.'' Unfortunately, the ``Theorem X'' in each theory is potentially as hard to prove as the four color theorem itself. In our case, the theorem is a statement about the non-vanishing of the  $E_\infty$ page of a spectral sequence (see Theorem C and Theorem F of \cite{BMcC}). The proof starts with an already non-vanishing  Khovanov-like homology theory, our bigraded $4$-color homology (the $E_1$ page), and then one needs to show that  the filtered $4$-color homology  (the $E_\infty$ page) does not vanish for all bridgeless planar graphs. 

To embolden a case for a non-computer-based proof, we prove that the smallest of the unavoidable  configurations -- the Birkhoff diamond -- is reducible using our filtered $n$-color homologies.  Since the Birkhoff diamond is historically and still-today the unavoidable configuration that  ``proves the rule,'' proving it reducible means that the $E_\infty$ page of our spectral sequence should be powerful enough to prove the remaining  unavoidable configurations are also reducible, leading to a new proof of the four color theorem. 

This has important consequences to gauge theorists' attempts at proving the four color theorem. Each group has been separately concentrating on their ``Theorem X''  with the idea to make progress on it outside of the accomplishments of notable mathematicians like Kempe, Heawood, Wernicke, Birkhoff, Franklin, Whitney, Lebesgue, and Heesch -- work that led up to  Appel and Haken's proof (cf. \cite{Wilson, Wilson4CS, WWP} for history of this problem). Instead, for the purposes of this paper, we stay within this well-established framework. In particular, we build off of the proof by Robertson, Sanders, Seymour, and Thomas  that uses $32$ discharging rules to build an unavoidable set of $633$ configurations \cite{RSST}. 

All three current valid proofs of the four color theorem  \cite{AppelHaken,AppelHaken2,RSST,Stromquist} break up into two steps: (1) build a list of unavoidable configurations and (2) show that each one of them cannot show up as a configuration in a minimal counter example. The contribution of this paper, based upon \cite{BaldCohomology} and \cite{BMcC}, is a completely new way to prove step (2). Since our method involves new techniques from our homology theories, the method can be used to attack smaller configurations than the Birkhoff diamond and thus possibly reduce the currently smallest unavoidable set of 633 configurations significantly down to human computable (see \Cref{Conjecture:Kempe-state-reducible}).   As Wilson explained in \cite{Wilson}, 

\begin{quote}\it
``For a non-computer proof, new ideas are needed, and such ideas have not yet been forthcoming.'' 
\end{quote}
The hope is that the ideas of this paper will accelerate the exploration of the four color problem using {\em all}  four gauge theory approaches  above.

The two main features of our method are (1) a (quantum) state system of orientable and non-orientable, often-higher-genus  surfaces that have the graph embedded into each of them (based on \cite{BaldCohomology}), and (2) the notion of state-reducibility within that system (based on \cite{BMcC}, see also \cite{BM-Vertex}).  The key distinction between our method and the standard method of proving the four color theorem is that our proof  \emph{never invokes the famous ``Kempe switch'' to prove reducibility}.

Recall the definitions of a Kempe switch and reducibility. First, suppose a planar trivalent graph $\Gamma$ has a certain configuration (a set of connected faces). Let a new graph $\Gamma'$ be constructed by removing this configuration and replacing it with a smaller configuration. Suppose that the new graph can be $4$-face colored. The smaller replacement configuration and graph are both called a {\em reducer} for $\Gamma$.  If the reducer is also trivalent and planar, then it is called  {\em Type $C$}. If it is a vertex, then it is called {\em Type $D$} (see \Cref{fig:BirkhoffTypeCTypeD} for examples). If a $4$-face coloring of $\Gamma'$ can be extended into the original configuration to get a $4$-face coloring of the original graph $\Gamma$, then the coloring is called {\em extendible}.  Bigons and triangles  are well known examples of configurations for which every $4$-face coloring is extendible \cite{Kempe}. 

\begin{figure}[h]
\includegraphics[scale=0.30]{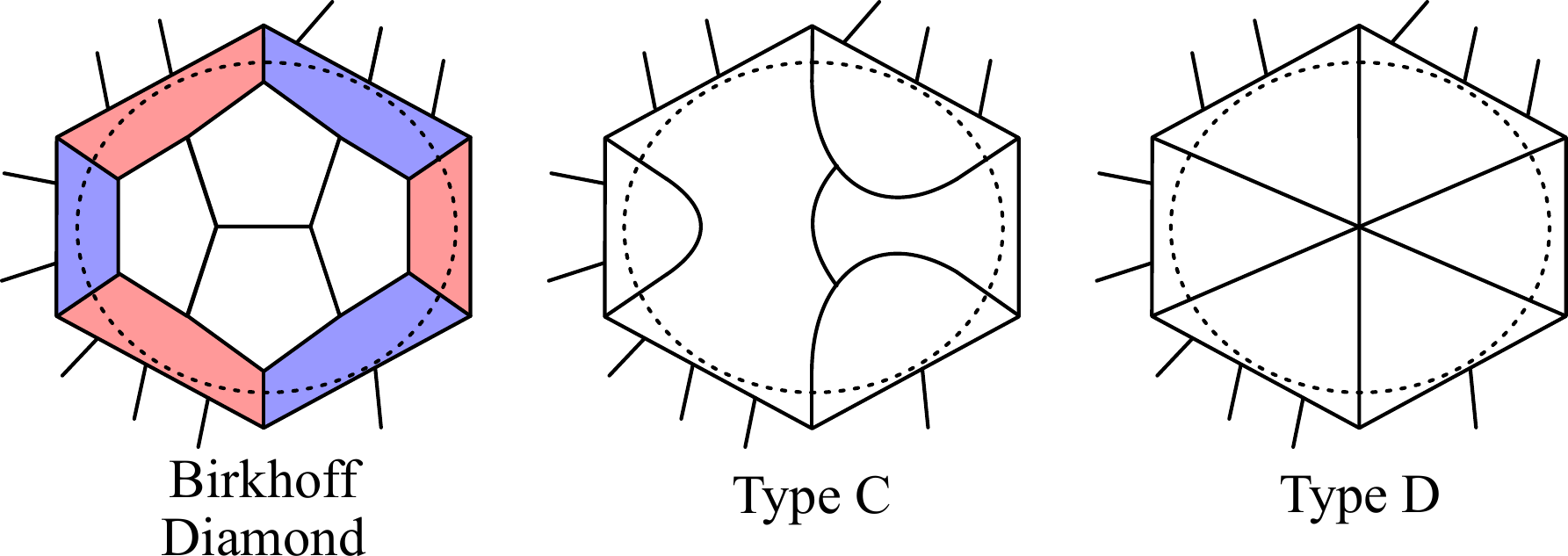}
\caption{The Birkhoff diamond inside an outer ring of six colored faces together with examples of Type $C$ and Type $D$ reducers for it.}
\label{fig:BirkhoffTypeCTypeD}
\end{figure}

Sometimes a coloring of $\Gamma'$ can only be extended to the outer ring of faces of the configuration in $\Gamma$ but fails to be extendible into the configuration (see the ring of colored faces in \Cref{fig:BirkhoffTypeCTypeD} for an example).  Suppose there is such a coloring where the configuration faces are left uncolored. For that coloring, assume there is a maximal connected set of  colored faces using only two colors from the coloring and that at least one of these faces is in the outer ring. A {\em Kempe switch} interchanges these two colors of this maximal set of faces to get a new coloring  \cite{MT}. Sometimes, if the Kempe switch is chosen carefully, this new coloring can then be extended to a $4$-face coloring of the original graph. A quadrilateral using a Type $D$ replacement is an example of a configuration that is only extendible after a Kempe switch: colorings that use all $4$ colors around the ring are not directly extendible. A {\em reducible configuration} is a configuration for which every coloring is extendible, either directly or after a succession of Kempe switches.  Note that a reducible configuration cannot appear in a minimal counterexample to the four color theorem. 

Bigons, triangles, and quadrilaterals are all considered trivially reducible configurations.  Kempe showed, using a simple Euler characteristic argument, that at least one of a bigon, triangle, quadrilateral, or a pentagon face must appear at least once in a planar trivalent graph \cite{Kempe}.  This set of four polygons became the first {\em unavoidable set} (cf. \cite{Wilson}).  Kempe thought he proved that the pentagon was reducible, and thus claimed to have proven the four color theorem, but eleven years later Heawood found a critical mistake in his argument \cite{Heawood}. Early attempts at the four color theorem replaced the nontrivial pentagon face with more configurations, each with more faces, that were still unavoidable \cite{Wernicke, Franklin}.  However, these configurations, which include examples like two adjacent pentagons and three pairwise-adjacent pentagons, are impossible to prove reducible using Kempe switches.  In 1913, Birkhoff introduced and proved the first nontrivial configuration (see \Cref{fig:BirkhoffTypeCTypeD}) that was reducible using Kempe switches:

\begin{theorem}[Birkhoff \cite{Birkhoff}] The Birkhoff diamond is a reducible configuration.
\end{theorem}

Since then, proofs of reducibility in all valid proofs of the four color theorem have used Kempe switches. Our idea is to replace the Kempe switch with a state system based upon a topological quantum field theory (cf. Section 9 of \cite{BMcC}) and then show how to extend solutions  (the colors) of a discrete ``color Dirac operator'' on the reducer of the graph to solutions on the original graph. We briefly describe this state system next; it is explained in detail in \cite{BMcC}.

Let $\Gamma$ be a trivalent plane graph of a graph $G(V,E)$ with vertices $V$ and edges $E$. This is a CW structure for $S^2$ whose $1$-skeleton is the graph.  Remove a small open disk from each $2$-cell of $\Gamma$ to get a ribbon graph (cf. \cite{BMcC} or \cite{Moffat2013}). This ribbon graph can be thought of as gluing closed ribbons (edges) to parts of boundaries of closed disks (vertices) so that the space is homeomorphic to the original ribbon graph.  For a planar graph, this can be done so that the disks and ribbons embed into the plane.  Each face corresponds to the boundaries of the ribbon graph and the ribbon graph can be described by these boundary components, see column $0$ of \Cref{fig:ThetaCube} for an example of the boundary outline of a ribbon theta graph.

Label and order the edges $\{e_1, \ldots, e_{|E|}\}$.  For an embeddable CW structure of a ribbon graph $\Gamma$ in the plane, construct a new ribbon graph as follows. Let  $\alpha \in \{0,1\}^{|E|}$ such that $\alpha=(\alpha_1,\ldots, \alpha_{|E|})$. For every $1\leq i \leq |E|$, if $\alpha_i=1$, remove the ribbon corresponding to edge $e_i$ and replace it with a ribbon with a half-twist. If $\alpha_i=0$, then do nothing.  The resulting ribbon graph is labeled $\Gamma_\alpha$ and called a {\em state graph} (cf. Section 6.6 of \cite{BMcC}).   These $2^{|E|}$ ribbon graphs can be arranged into a hypercube with an edge between two ribbon graphs if and only if the $\alpha_i$'s differ for one $i$. An example of the theta graph hypercube is displayed in \Cref{fig:ThetaCube} where the immersed circles represent the boundaries of the faces. One can glue disks (faces) to the boundaries of $\Gamma_\alpha$ to get an associated closed surface, which will still be called  $\Gamma_\alpha$. Sometimes this surface is non-orientable. The original graph $G(V,E)$ embeds into each state graph $\Gamma_\alpha$ as the $1$-skeleton of the CW structure. A (proper) {\em $n$-face coloring} of $\Gamma_\alpha$ is a coloring of the circles (boundary of faces) with $n$ distinct colorings such that no two faces adjacent along an edge share the same color.

\begin{figure}[H]
\includegraphics[scale=0.30]{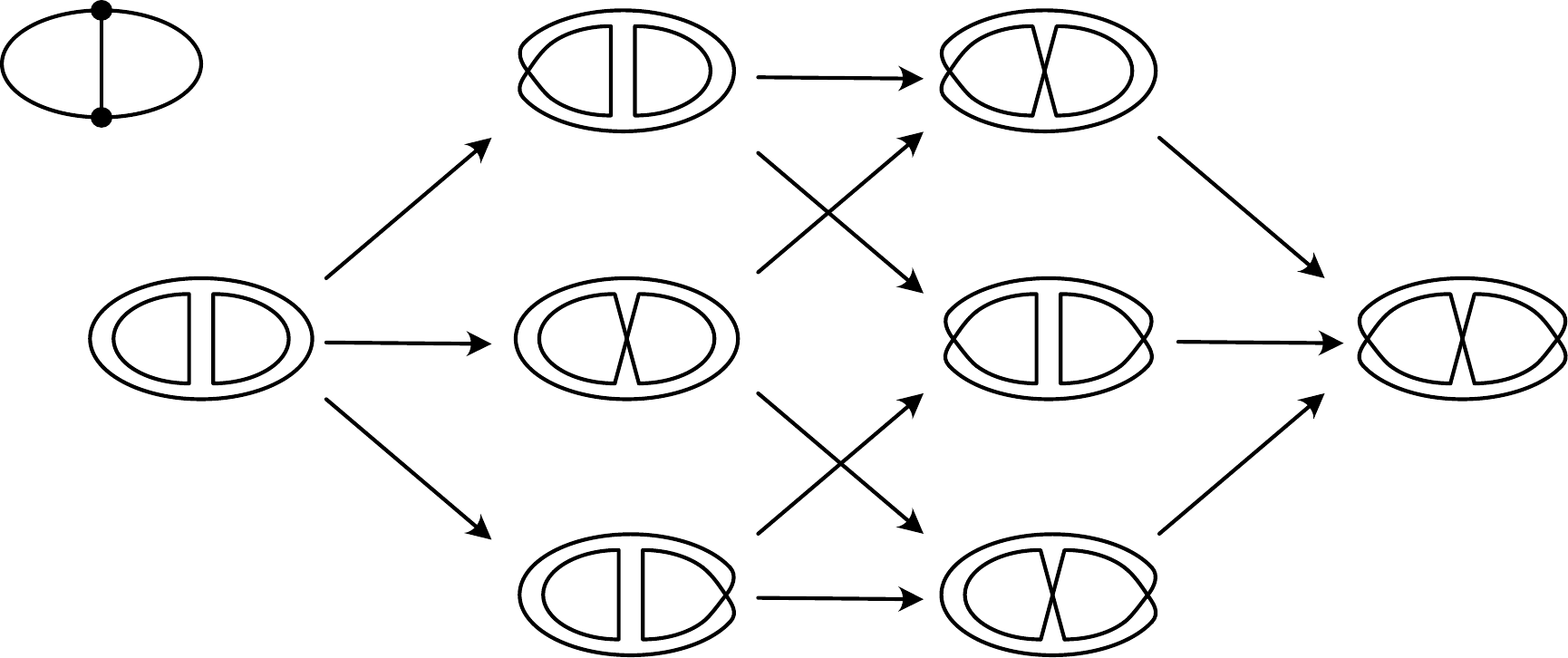}
\caption{The hypercube of states of the theta graph. Columns are labeled by the number of half-twists, $|\alpha|=\alpha_1+\cdots + \alpha_{|E|}$, therefore the first column is column 0, the second is column 1, etc., and represent the degree of the homology. The closed surfaces in column 1 are $\BR P^2$'s, column 3 is a $T^2$.}
\label{fig:ThetaCube}
\end{figure}

For a given positive integer $n$, it was proven in \cite{BMcC}, Theorem D, that the filtered $n$-color homology $\widehat{CH}_n^i(\Gamma)$ is a vector space generated by all possible $n$-face colorings on all state graphs $\Gamma_\alpha$ such that $|\alpha|=i$, i.e., it is generated by $n$-face colorings on all ribbon graphs associated to $\Gamma$ with the same number of half-twists in them (see the columns in \Cref{fig:ThetaCube}). In that paper, these colorings are derived from solutions of a Hodge-like Dirac operator on the complex of state graphs.   

One can think of this homology theory, then, as a state system generated by ``enhanced states,'' that is, ordered pairs $(\Gamma_\alpha, c)$ where $c$ is an $n$-face coloring of the state graph $\Gamma_\alpha$. Using enhanced states, we can give a definition of state-extendible colorings of a graph from its reducer. Thus, while the proofs in this paper all use the filtered $n$-color homologies of \cite{BMcC} to get  correspondences between solutions of a Dirac operator and $n$-face colorings, we can and will ``hide'' this homology theory here except only to mention what the gluing theorems look like in \Cref{theorem:gluing-map}.  Thus, we speak only of graphs and colorings on graphs to state and prove our main theorems. 

To setup these definitions, let $\Gamma$ be a plane graph of a bridgeless trivalent connected planar graph $G(V,E)$ and $C$ be a configuration in $\Gamma$.  Let $C'$ be a configuration that can replace $C$ in $\Gamma$  to get a new plane graph $\Gamma'$ that is also connected  and trivalent. If the sum of vertices and edges of $\Gamma'$ is less than the sum of vertices and edges of $\Gamma$, then $\Gamma'$ is a {\em reducer} of $\Gamma$. The Type $C$ configuration in \Cref{fig:BirkhoffTypeCTypeD} is an example of a reducer of the Birkhoff diamond. In fact, it is {\em the} reducer  we will use in this paper.

\begin{definition}
Suppose that $(\Gamma'_\alpha, c')$ is an enhanced state of a reducer $\Gamma'$ where $c'$ is an $(n-1)$-face coloring of a state graph $\Gamma'_\alpha$ of $\Gamma'$. The enhanced state is called {\em state-extendible} if there exists an enhanced state $(\Gamma_\beta, c)$ of the original graph $\Gamma$ where $c$ is an $n$-face coloring of $\Gamma_\beta$ and both states are the same outside of their configurations.
\end{definition}

State-extensions are likely: even though an $(n-1)$-face coloring on $\Gamma
'_{\alpha'}$ may not directly extend to an $(n-1)$-face coloring on some $\Gamma_\beta$, with one extra color, it often does.  It is also more likely to extend since there are many more states on $\Gamma$ than $\Gamma'$ that match on the complement of the configurations. It should also be noted that an extension does not, in general, correspond to a Kempe switch in any of its forms. There are obvious notions of Kempe switches for a given $n$-face coloring of a surface, but these notions are not needed for this paper. They will become important in future research (see \Cref{Conjecture:Kempe-state-reducible}).

The definition of state-extendible is general once a state system is understood for nonplanar ribbon graphs (see \cite{BMcC}). Of course, this paper is concerned with the planar, $n=4$ case:

\begin{definition}
Let $C'$ be a reducer for a configuration $C$ of $\Gamma$. Suppose that the set of enhanced states $(\Gamma'_\alpha,c')$ of $3$-face colorings for $\Gamma'$ is nonempty. If every enhanced state from this set is state-extendible to $(\Gamma_\beta, c)$ where $c$ is a  $4$-face coloring on some state graph $\Gamma_\beta$ of $\Gamma$, then the configuration $C$ is said to be {\em state-reducible} with reducer $C'$.
\end{definition}

When comparing this definition to the original definition of reducibility, this definition analogous to  {\em directly} extendible -- no succession of ``Kempe switches'' are introduced or needed. Equipped with these definitions, we can state the main theorem of this paper:

\medskip
\begin{Theorem} \label{Theorem:State-Reducible}
Let $\Gamma$ be a connected plane trivalent graph. Suppose that $\Gamma$ is state-reducible with reducer $\Gamma'$. If $\Gamma'$ is $4$-face colorable, then $\Gamma$ is $4$-face colorable.
\end{Theorem}
\medskip

State-reducibility is our ``new idea'' in Wilson's quote above.  Note that there is an interplay between $4$-face colorings on the planar trivalent graph of the reducer and the existence of $3$-face colorings on higher genus state graph surfaces of the reducer. There is also a relationship between $4$-face colorings on the state graphs and $4$-face colorings on the original graph. These results make up the contents of \Cref{theorem:4-to-3-face-coloring} and \Cref{theorem:4-face-to-4-face}, both of which are derived from the filtered $n$-color homology theories in \cite{BMcC}. 

To justify our claim that this new idea is useful, we must also show that \Cref{Theorem:State-Reducible} can be used in at least one, nontrivial,  example.  As an application of \Cref{Theorem:State-Reducible}, we prove that the Birkhoff diamond is state-reducible:

\begin{Theorem} \label{Theorem:Main} The Birkhoff diamond,
\parbox[m]{2cm}{\includegraphics[scale=0.20]{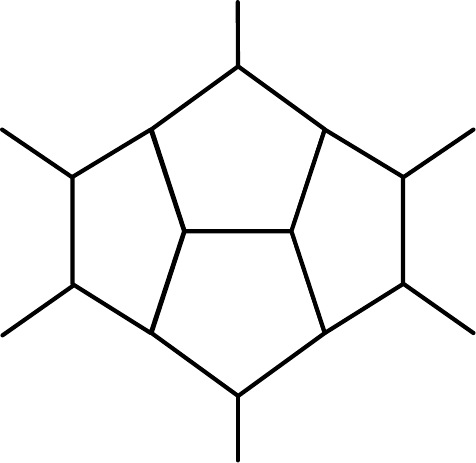}},
is state-reducible with  reducer,
\parbox[m]{2cm}{\includegraphics[scale=0.20]{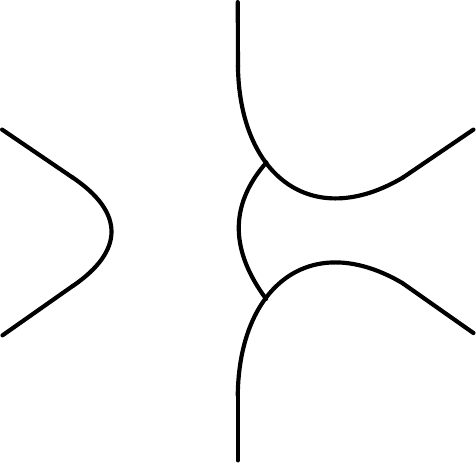}}.
\end{Theorem}

\Cref{Theorem:State-Reducible} and  \Cref{Theorem:Main} combine to form the first new proof that the Birkhoff's diamond is reducible (in the original sense) in 111 years.

While \Cref{Theorem:State-Reducible} is the theory behind this application,  \Cref{Theorem:Main} is based upon computations on a supercomputer. At this point, we could do the same analysis outlined in \Cref{proposition:3-face-to-4-face-Birkhoff}  on the remaining $632$ unavoidable configurations using \Cref{Theorem:State-Reducible} and thereby complete the proof of the four color theorem using filtered $3$- and $4$-color homology. This would take planning, access to a more powerful supercomputer, and work on making our current program more efficient.  But nothing new in terms of theory would be needed.

However, there is a much more enticing research program to pursue using our new techniques.  The Birkhoff diamond is the smallest configuration in every valid proof of the four color theorem. In fact, it has been shown that one {\em cannot prove} configurations that are smaller, like a single pentagonal face, are reducible in the original sense, mainly because there are no Kempe switches for particular colorings (see \cite{AppelHaken0}, who cite \cite{Bernhart} and \cite{Heesch}, and ultimately \cite{Birkhoff} for how this is proven).  Since state-reducibility does not use Kempe switches to prove the smallest nontrivial configuration reducible, we expect it to be {\em applicable to smaller configurations} by including Kempe  switches, which allow far more flexibility in showing extendibility. Call a configuration {\em Kempe state-reducible} if every $3$-face coloring on a state graph of the reducer is state-extendible to a $4$-face coloring on some state graph of the original, either directly as in \Cref{Theorem:Main} or by a succession of Kempe switches on colorings on state graph surfaces. We conjecture:

\begin{Conjecture}\label{Conjecture:Kempe-state-reducible}
At least one of the following configurations is Kempe state-reducible: a single pentagon, a single hexagon, two adjacent pentagons, or three pairwise-adjacent pentagons.  
\end{Conjecture}

If the configuration of three pairwise-adjacent pentagons is Kempe state-reducible, then it cannot show up in a minimal counterexample, and $148$ of the $633$ unavoidable configurations in \cite{RSST} could be replaced by this configuration (since it shows up as a sub-configuration in them). This would be an almost 25 percent reduction in the size of the current smallest unavoidable set of reducible configurations. It would also open the door for proving that Franklin's unavoidable set of six nontrivial configurations are reducible (cf. \cite{Franklin}). Obviously, if the hexagon or double pentagon configurations were Kempe state-reducible, this would shrink the size even more drastically (cf. \cite{Wernicke}). Finally, if the pentagon was Kempe state-reducible, then this would complete Kempe's original proof of the four color theorem. It would be a non-computer-based proof since $3$-colorings on state graphs of  possible reducers can be enumerated by hand, in fact, in a very short list. We expect at this stage  one would need our bigraded theory and the spectral sequence of \cite{BMcC}. Building the tools (cf. \Cref{theorem:gluing-map}) needed for this project is future research.

\section{Proof of \Cref{Theorem:State-Reducible}}

\Cref{Theorem:State-Reducible} follows from \Cref{theorem:4-to-3-face-coloring} and \Cref{theorem:4-face-to-4-face} below,  as summarized by  \Cref{fig:TheoremAProofPlan}. Let $\Gamma$ be a connected trivalent plane graph containing some configuration $D$. Suppose that there is another,  trivalent plane configuration $D'$ that can be inserted into $\Gamma\setminus D$ to get $\Gamma'$. Suppose that $\Gamma'$ is $4$-face colorable and that $\Gamma$ is state-reducible with reducer $D'$. Then there exists a $4$-face coloring of $\Gamma$ (\Cref{Theorem:State-Reducible}) by following the clockwise path  in \Cref{fig:TheoremAProofPlan}:

\begin{figure}[h]
$$\begin{tikzpicture}[thick, scale = .8]
\draw (0,0) node {A $4$-face coloring on $\Gamma'$};
\draw (3,0) -- (6.2,0);
\draw (3,-0.16) -- (6.2,-0.16);
\draw (6.1,0.1)--(6.3,-0.08)--(6.1,-0.26);
\draw (4.5,.6) node {\Cref{theorem:4-to-3-face-coloring}};             
\node[text width = 4.7 cm] at (10,0) {A $3$-face coloring on some \\ state graph $\Gamma_{\alpha}'$ of $\Gamma'$};
\begin{scope}[yshift = -4cm]
\draw (0,0) node {A $4$-face coloring on $\Gamma$};
\begin{scope}[xscale=-1,xshift=-9.2cm]
\draw (3,0) -- (6.2,0);
\draw (3,-0.16) -- (6.2,-0.16);
\draw (6.1,0.1)--(6.3,-0.08)--(6.1,-0.26);
\end{scope}
\draw (4.5,.6) node {\Cref{theorem:4-face-to-4-face}};             
\node[text width = 4.8 cm] at (10,0) {A $4$-face coloring on a related state graph $\Gamma_{\beta}$ of $\Gamma$};
\end{scope}
\draw (0,-1) -- (0,-3);
\draw (0.16,-1) -- (0.16,-3);
\draw (-0.1,-2.9)--(0.08,-3.1)--(0.26,-2.9);
\draw (-2,-2) node {\bf \Cref{Theorem:State-Reducible}};  
\begin{scope}[xshift = 9cm]
\draw (0,-1) -- (0,-3);
\draw (0.16,-1) -- (0.16,-3);
\draw (-0.1,-2.9)--(0.08,-3.1)--(0.26,-2.9);
\draw (2,-2) node {\hspace{.6cm}\parbox[m]{3cm}{Definition of\\state-reducibilty}};  
\end{scope}
\end{tikzpicture}$$
\caption{A diagram of the proof of \Cref{Theorem:State-Reducible} for a plane graph $\Gamma$ with a configuration and its reducer $\Gamma'$ with a $4$-face coloring.}
\label{fig:TheoremAProofPlan}
\end{figure}
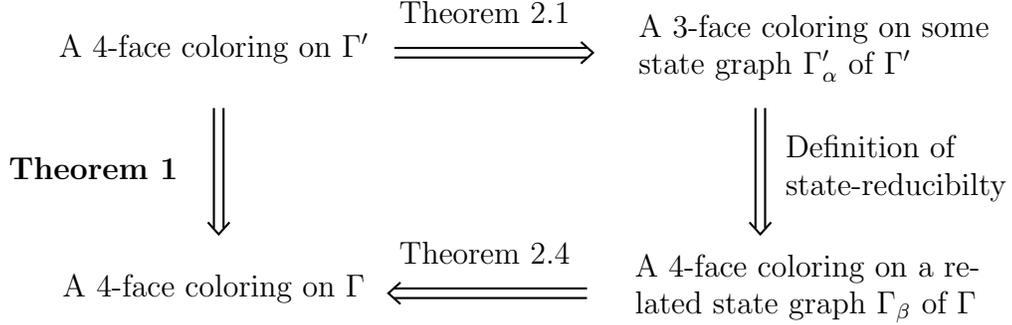

\begin{theorem}\label{theorem:4-to-3-face-coloring}
Let $\Gamma$ be a plane graph of a connected trivalent planar graph $G(V,E)$.  Then
$$ \#\left\{\mbox{$4$-face colorings of $\Gamma$}\right\} = 4\cdot \sum_{\alpha\in\{0,1\}^{|E|}} \# \left\{\mbox{$3$-face colorings of $\Gamma_\alpha$}\right\}.$$
In particular, every $4$-face coloring of $\Gamma$ corresponds in a $4$-to-$1$ way to a $3$-face coloring of some state graph $\Gamma_\alpha$ of $\Gamma$.
\end{theorem}

\begin{proof}
By Theorem~F.3 in \cite{BMcC}, the sum of the dimensions of the filtered $3$-color homology is equal to the number of $3$-edge colorings of $G$, i.e.,
\begin{equation}\label{eq:3-edge-color-equals-sum}
\# \left\{\mbox{$3$-edge colorings of $G$}\right\} = \sum_{i=0}^{|E|} \dim \widehat{CH}_3^i(\Gamma;\BC).
\end{equation}
See Theorem~8.9 in Section 8.3 of \cite{BMcC} for the full proof.  \Cref{eq:3-edge-color-equals-sum} follows from the fact that the total face color polynomial, $T(G,n)$, used in the statement of Theorem~F.3, is based upon the Poincar\'{e} polynomial of the filtered $n$-color homology. Therefore, $T(G,3)$ is equal to the righthand side of \Cref{eq:3-edge-color-equals-sum}.  By  Theorem~D.1 of \cite{BMcC},  $\widehat{CH}^i_3(\Gamma;\BC)$ is isomorphic to the direct sum of the spaces of harmonic colorings on the states $\Gamma_\alpha$ where $i=|\alpha|$, and by Theorem~D.2, the dimension of the harmonic colorings for $\Gamma_\alpha$ is equal to the number of $n$-face colorings of $\Gamma_\alpha$. Thus, $T(G,3)$ is equal to the sum of the number of $3$-face colorings on each state graph $\Gamma_\alpha$ of $\Gamma$, summed over all $\alpha\in\{0,1\}^{|E|}$.

Next, we show that there is a one-to-one correspondence between $3$-edge colorings of $\Gamma$ and $3$-face colorings of state graphs of $\Gamma$, and importantly, the set of state graphs is {\em exactly} the collection of surfaces where one should look for this correspondence. 

The straightforward direction of the correspondence is from a $3$-edge coloring on $\Gamma$ to a $3$-face coloring on some state graph. Given a plane diagram $\Gamma$ with a $3$-edge coloring, a calculation shows that there is a unique $3$-edge coloring on the blowup of $\Gamma$ for that coloring (cf. Definition 2.11 of \cite{BMcC}). From the $3$-edge coloring of the blowup, one can assemble a state-graph with a $3$-face coloring using the process shown in \Cref{fig:From-3edge-to-3-face} for each edge.

\begin{figure}[H]
\includegraphics[scale=0.35]{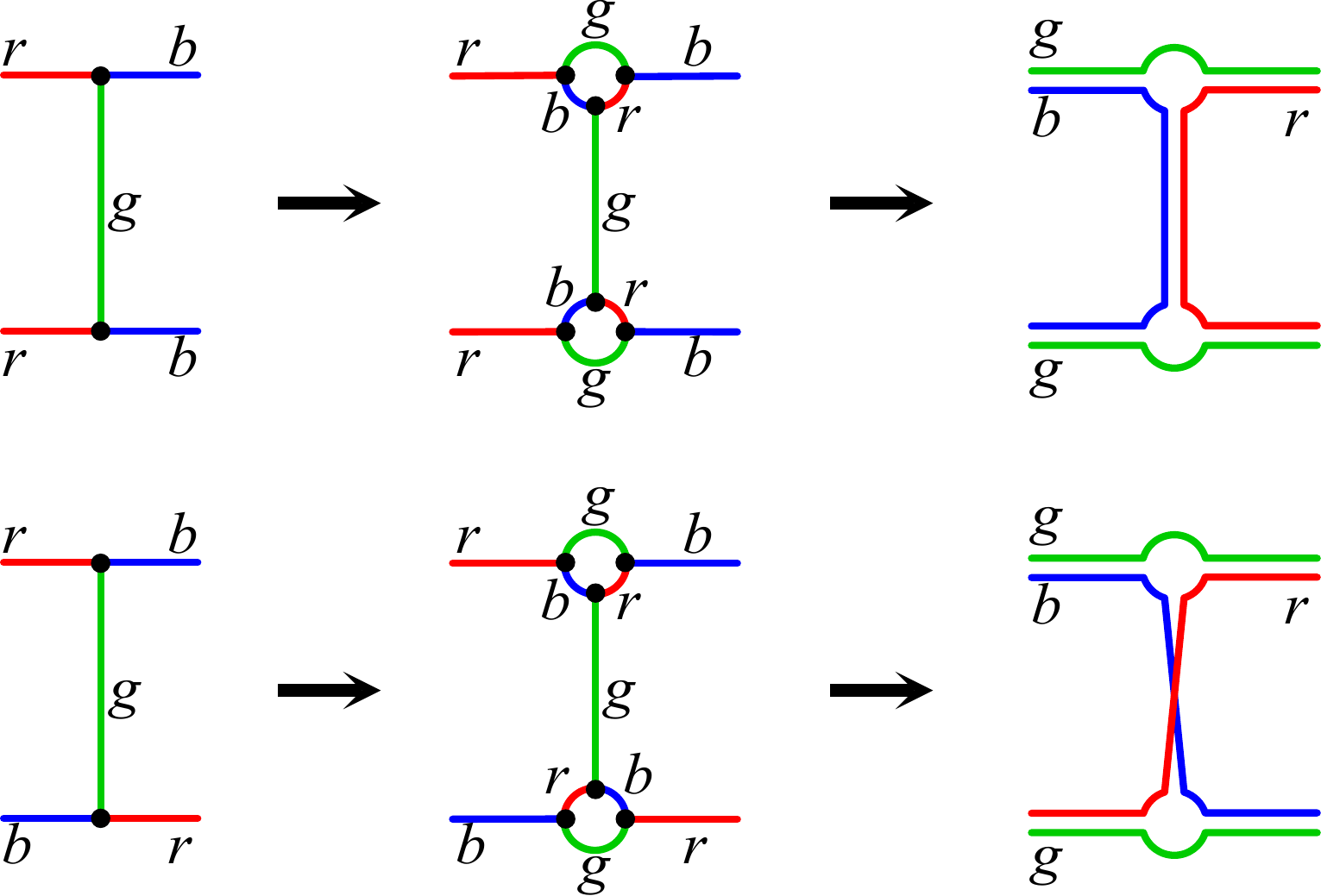}
\caption{From a $3$-edge coloring of $\Gamma$ at an edge to a $3$-face coloring on a state graph of $\Gamma$.}
\label{fig:From-3edge-to-3-face}
\end{figure}

The reverse direction of the correspondence is much harder and uses the ``Poincar\'{e} Lemma'' of Proposition~6.8 of \cite{BMcC}.  This is the main theorem needed to prove Theorem~D.1; in fact, it is one of the most important and fundamental ideas of that paper. This theorem ensures that we are not ``over counting'' $3$-edge colorings in hypercube of states, that is, this correspondence is one-to-one and onto the $3$-face colorings in the state system. 

Using the entire set of state graphs is important because the state system is used in the calculation of the  Penrose polynomial at $n=3$. This evaluation is the Euler characteristic of the filtered $3$-color homology. Penrose showed in 1971 in \cite{Penrose} that the evaluation of the Penrose polynomial at $n=3$ of a plane trivalent graph was equal to the number of $3$-edge colorings of it.  We have been using the total face color polynomial of Theorem~F.3 in this paper, but for planar graphs, the Penrose polynomial and total face color polynomial are equal by Theorem~6.11 of \cite{BMcC}.    Thus, this proves that the righthand side of \Cref{eq:3-edge-color-equals-sum} is equal to the lefthand side.

The final step in the proof of \Cref{theorem:4-to-3-face-coloring} is to note that Tait used the Klein four-group to prove that the number of $4$-face colorings of $\Gamma$ is equal to four times the number of $3$-edge colorings of $\Gamma$ when $\Gamma$ is trivalent and planar \cite{Tait}.  This idea has been discussed in many papers, including, for instance, Example~3.7 of \cite{BMcC}. 

Since every step discussed above was done at the level of a correspondence (not just a count), the theorem is proved, including the top arrow in \Cref{fig:TheoremAProofPlan}.
\end{proof}

\begin{example}
\Cref{fig:ThetaCube} is an instructive example of \Cref{theorem:4-to-3-face-coloring}: The number of $3$-face colorings of the state graphs of the theta graph is six, which are all found on the original graph in Column 0, i.e., $\widehat{CH}^0_3(\theta) \cong \BC^6$. This is also equal to the number of $3$-edge colorings of the theta graph. Note that it requires all eight state graphs to compute the Penrose polynomial, $P(\theta,n) = n^3 - n^2 - n^2- n^2 +n+n+n-n$, even though seven of them do not contribute actual colorings in the homology.
\end{example}

In all pictures that follow, red, blue, and green will be used for $3$-face colorings, and red, blue, green, and yellow will be used for $4$-face colorings.

\begin{remark} The fact that the total face color polynomial evaluated at $n=3$ counts the number of $3$-edge colorings for non-planar graphs indicates  that many of the ideas in this paper can be extended to the non-planar setting (cf. \cite{BKM}). 
\end{remark}

The bottom arrow of \Cref{fig:TheoremAProofPlan} also follows directly from the correspondences between $n$-face colorings and generators of filtered $n$-color homology in \cite{BMcC}.  

\begin{theorem}\label{theorem:4-face-to-4-face}
Let $\Gamma$ be a plane graph of a connected trivalent planar graph $G(V,E)$.  If a state graph $\Gamma_\alpha$ of $\Gamma$ is $4$-face colorable, then $\Gamma$ is $4$-face colorable.
\end{theorem}

\begin{proof} In terms of filtered $n$-color homology, the theorem follows immediately from Theorem~F.2, Theorem~D.1, and Theorem~D.2 of \cite{BMcC} for $n=4$. Ultimately, the idea behind the proof is a generalization of Tait's original argument using filtered $4$-color homology -- see Theorem~8.5 of \cite{BMcC} for details.
\end{proof}

\section{Combinatorial gluing theorems}\label{section:combinatorial-gluing-theorems}

In this section, we briefly describe the process of finding a state-extension using an example. We then write out what this looks like using a relative version of filtered $n$-color homology and write down a gluing theorem that will be explored in future research.

Suppose we have the graph shown in \Cref{fig:BirkhoffExampleColor0} with a Birkhoff diamond and we have replaced it with its reducer from \Cref{Theorem:Main}.

\begin{figure}[H]
\includegraphics[scale=0.40]{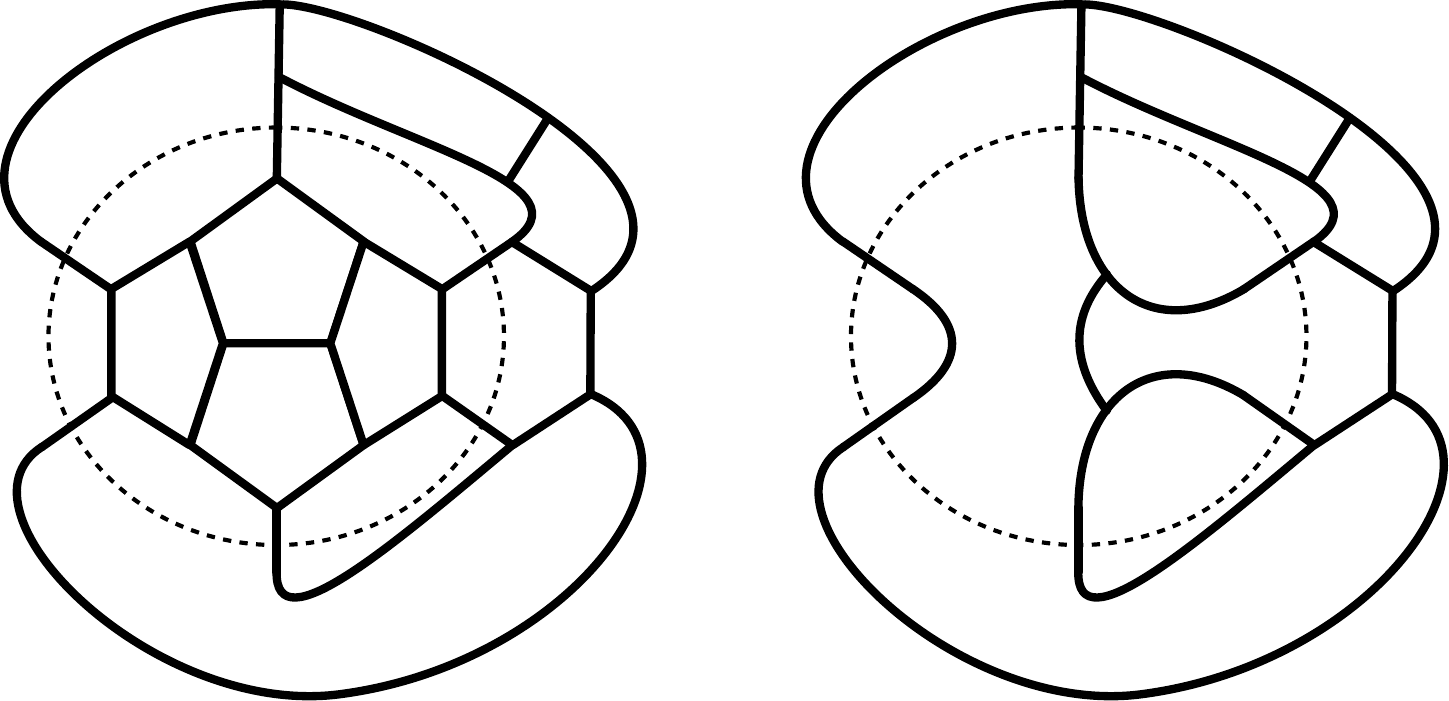}
\caption{A graph with a Birkhoff diamond and its reducer.}
\label{fig:BirkhoffExampleColor0}
\end{figure}

Next, suppose there is a $3$-face coloring of some state graph of the reducer. \Cref{fig:BirkhoffExampleColor1} shows an example of such a coloring on the left.  Two circles that are adjacent at an edge are called {\em interacting}. A proper face coloring requires interacting circles to be different colors. Note that the state graph is given by the immersed circles as was described for the surfaces in \Cref{fig:ThetaCube}.  In this case, the $3$-face coloring  on the left of  \Cref{fig:BirkhoffExampleColor1} is a torus with $1$-skeleton the reducer.

Transfer this coloring to the original graph with the Birkhoff diamond except for the parts of the faces that are part of the reducer. This will color all the circles (faces) away from the configuration in the original graph and the six arcs (partial faces) that enter the Birkhoff diamond, as shown on the right in \Cref{fig:BirkhoffExampleColor1}. 

\begin{figure}[H]
\includegraphics[scale=0.40]{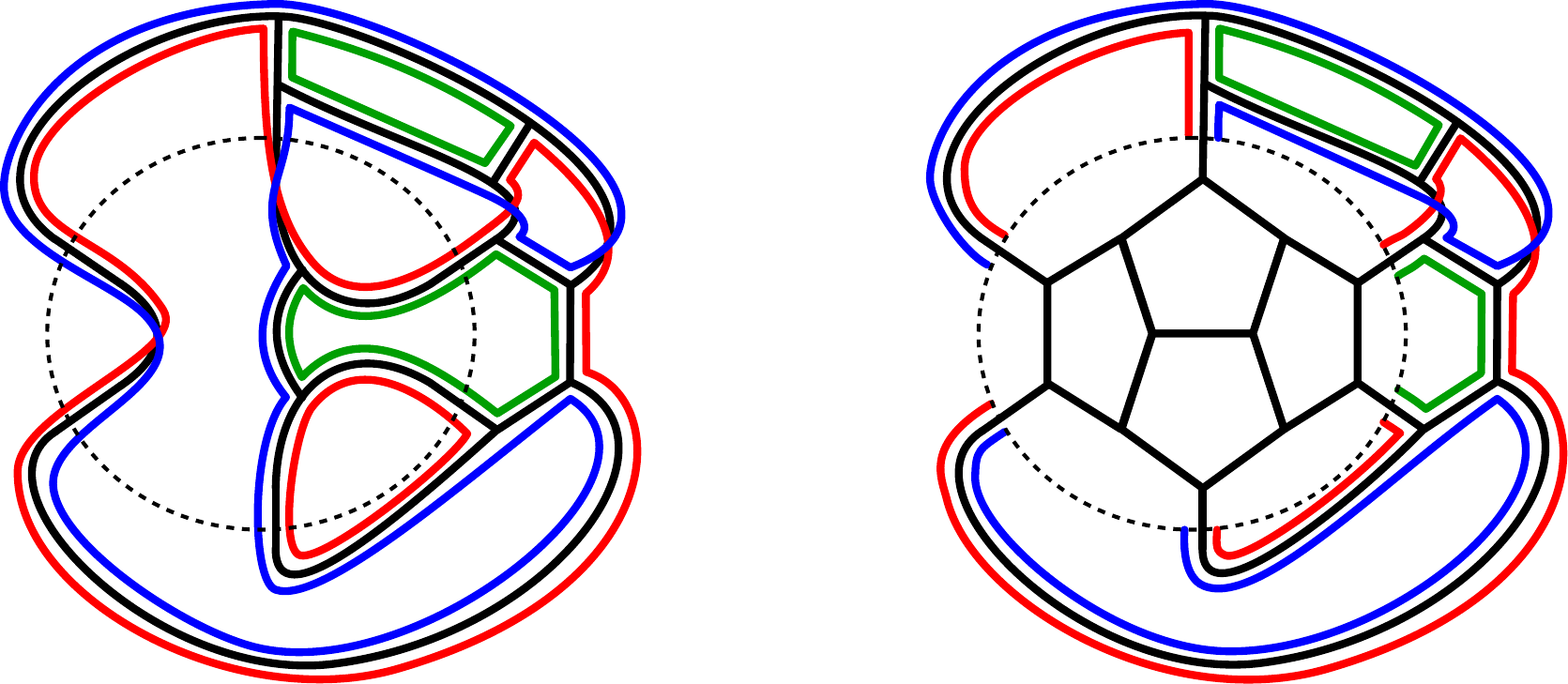}
\caption{A $3$-face coloring of the reducer and the start of a $3$-face coloring of the original graph using that coloring.}
\label{fig:BirkhoffExampleColor1}
\end{figure}

The goal is to extend the colorings on the arcs into the Birkhoff diamond. Ignoring the colors for a moment, there are $2^{21}$ ways the arcs can be extended into the Birkhoff diamond.  The left picture in \Cref{fig:BirkhoffExampleColor2} is one such extension.  Initially, the $3$-face coloring extension may not be a proper coloring -- it may have two  adjacent  faces with the same color, or it may connect two differently colored arcs together to form a circle (see \Cref{fig:BirkhoffExampleColor2} again with circle formed from a red and blue arc). If either of these types of colorings show up in the $2^{21}$ ways, a proper $4$-face coloring can often be obtained from it by painting one (or more) of the circles by a fourth color, as shown with the yellow circle in the right picture of \Cref{fig:BirkhoffExampleColor2}.

\begin{figure}[H]
\includegraphics[scale=0.40]{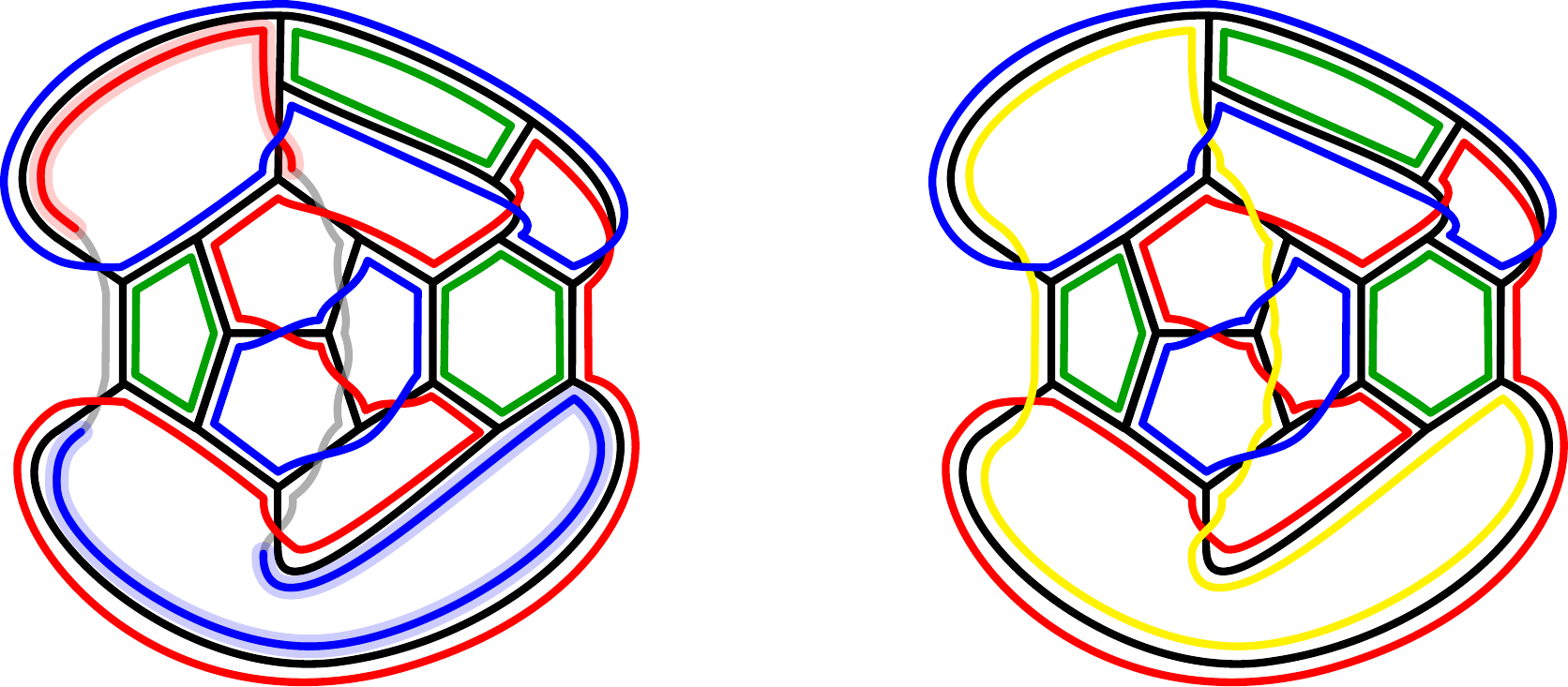}
\caption{The left picture is one possible state with (almost) a $3$-face coloring of the Birkhoff diamond from the $3$-face coloring of the reducer in \Cref{fig:BirkhoffExampleColor1}. Several of the circles in the Birkhoff configuration can be given appropriate colorings, but the highlighted red-blue arcs form a circle that cannot be properly colored red or blue (or green) to form a $3$-face coloring.  However, the  picture on the right is a proper $4$-face coloring of the original graph found by painting the red-blue circle by the fourth color, yellow.}
\label{fig:BirkhoffExampleColor2}
\end{figure}

If the coloring can be extended into the Birkhoff diamond and turned into a $4$-face coloring by coloring one or two circles with the fourth color, then by \Cref{theorem:4-face-to-4-face}, the $3$-face coloring of the reducer induces a $4$-face coloring on the original plane graph.  The proof of \Cref{Theorem:Main}  shows that this coloring process can always be done for every potential $3$-face coloring of the reducer. To capture every potential $3$-face coloring, we need to generate a complete list of all possible ways the six arcs can interact with each other in a graph minus the Birkhoff diamond region (cf. the arcs in the right picture of \Cref{fig:BirkhoffExampleColor1}).  Describing these ``caps'' and colorings on them is the content of the next section.

\begin{remark} \label{remark:face-vs-vertex-colorings} The reader can now see why we work with  face colorings instead of vertex colorings on the dual graph as is usually done in graph theory: The coloring process described above does not nicely translate into vertex colorings on the dual ribbon graphs of state graphs where the circles are replaced by vertices and interactions by edges. The $2^{21}$ ways of extending could, in principle, be converted into vertex colorings, but then some $4$-face colorings we construct later would be almost impossible to detect if dual graphs and vertex colorings were used. However, $4$-face colorings are  detectable using the state graphs. This is because it is possible, as in the example in  \Cref{fig:BirkhoffExampleColor2}, to see how to combine a red arc with a blue arc to get a circle that runs through the Birkhoff diamond, which is then painted the fourth color (yellow).
\end{remark}

We summarize the theory behind this section with a statement of a theorem that will be proven in later work in greater generality (cf. the statement after \Cref{Conjecture:Kempe-state-reducible}). For now, suppose that $\Gamma\setminus C$ is a ``ribbon graph with boundary'' found by removing a configuration $C$ from it.  Furthermore, assume that we have defined a ``relative filtered $n$-color homology,'' $\widehat{CH}_n^*(\Gamma\setminus C,\del (\Gamma\setminus C))$, that captures $n$-face colorings up to this boundary.  See the righthand picture of \Cref{fig:BirkhoffExampleColor1} for an example of a homology class in this homology. Of course, the configuration $C$ can also be thought of as  a ``ribbon graph with boundary.'' 

\begin{theorem} There exists a gluing map,
$$S_C:\widehat{CH}_n^i(\Gamma\setminus C,\del (\Gamma\setminus C)) \ot \widehat{CH}_n^j(C,\del C) \ra \widehat{CH}_n^{i+j}(\Gamma),$$
which describes when relative harmonic solutions of the discrete color Laplacian can be glued together to get a solution for the entire graph.
\label{theorem:gluing-map}
\end{theorem}

\begin{remark}The discrete color Laplacian should not be confused with the discrete Laplacian in graph theory. See Section 6.1 of \cite{BMcC} for its definition.
\end{remark} 

This theorem also holds for the space of harmonic colorings, $\widehat{\mathcal{CH}}_n(\Gamma_\alpha)$, for individual states $\Gamma_\alpha$ (cf. Section 6.4 of \cite{BMcC}). If $\Gamma_\alpha$ is a relative state graph of $\Gamma\setminus C$ and $C_\beta$ is a relative state graph of $C$, then let $\Gamma_{\alpha \# \beta}$ be the ``glued together'' state graph of $\Gamma$. Then the map becomes:
$$S_{{\alpha,\beta}}:\widehat{\mathcal{CH}}_n(\Gamma_\alpha, \partial \Gamma_\alpha ) \ot \widehat{\mathcal{CH}}_n(C_\beta, \partial C_\beta) \ra \widehat{\mathcal{CH}}_n(\Gamma_{\alpha \# \beta}).$$

\Cref{Theorem:Main} can be restated using this idea. It follows from:

\begin{proposition}
Let $BD$ be the Birkhoff diamond configuration and $R$ its reducer as in \Cref{fig:BirkhoffExampleColor0}.  For every relative state graph $\Gamma_\alpha$  of $\Gamma\setminus R$ and  $[\gamma]\in \widehat{\mathcal{CH}}_3(\Gamma_\alpha, \partial \Gamma_\alpha ) $ where  there exists a  relative state graph of the reducer, $R_\beta$, such that 
$$S_{R_\beta}([\gamma]): \widehat{\mathcal{CH}}_3(R_\beta, \partial R_\beta)  \ra  \widehat{\mathcal{CH}}_3(\Gamma_{\alpha \# \beta})$$
is a nonzero map, then there exists a relative state graph of the Birkhoff diamond, $BD_{\beta'}$, and a  $[\gamma'] \in\widehat{\mathcal{CH}}_4(\Gamma_\alpha, \partial \Gamma_\alpha)$ such that
$$S_{BD_{\beta'}}([\gamma']): \widehat{\mathcal{CH}}_4(BD_{\beta'}, \partial BD_{\beta'})  \ra  \widehat{\mathcal{CH}}_4(\Gamma_{\alpha \# \beta'})$$
is also a nonzero map.
\label{proposition:gluing-extension}
\end{proposition}

We have hidden the $n$-color homology theory behind \Cref{theorem:gluing-map} and \Cref{proposition:gluing-extension} to a point where it is purely a combinatorial process that can be  implemented on a computer (see \Cref{section:main-computation}). Therefore, we do not need this theorem for this paper and will work with enhanced states $(\Gamma_\alpha, c)$ instead. However,  gauge theorists should recognize \Cref{theorem:gluing-map} as a combinatorial version of a ``gluing theorem'' for relative $n$-color homology in the same way as one patches solutions of non-linear differential equations along a neck as in Seiberg-Witten theory  (cf. \cite{MST} and \cite{MMS}). Knot theorists should recognize this as working with  ``tangles'' as in  \cite{BN3}, but our tangles (the relative states $\Gamma_\alpha$ and $C_\beta$) glue together to get multicolored ``links'' $\Gamma_{\alpha\sharp\beta}$; imagine the picture on the left in \Cref{fig:TakingTheCap} with colors, for example.

\section{The theory behind how to prove \Cref{Theorem:Main}}
\label{section:taking-the-cap-of-a-state}

This section addresses the question of why one can restrict to a finite list of interacting arcs (the ``caps'') to prove \Cref{Theorem:Main}. The problem becomes apparent if one looks at the righthand picture of \Cref{fig:BirkhoffExampleColor1}.  In that small example, there are six arcs, four of which are adjacent along  edges with the green circle in the top right corner of the state. In general, however, for a graph with $|E| \gg 0$, one may expect all six arcs could interact with each other numerous times and with any number of circles. (The arcs can  spread out from the configuration ``deep'' into the graph.) Furthermore, the circles that interact with the arcs may interact with each other in a multitude of ways. Thus, there is potentially an unbounded number of states that would need to be  checked to prove \Cref{Theorem:Main}.

The state-extension process described in \Cref{section:combinatorial-gluing-theorems} is what forces this list to be finite and relatively small. First, we always start with a $3$-face coloring of a state graph of the reducer, which mean all circles outside of the reducible configuration are colored by red, blue, or green so that pairwise interacting circles are different colors. In finding a $4$-face coloring of the original graph, only one or two of the non-interacting arcs are potentially painted the fourth color (compare  \Cref{fig:BirkhoffExampleColor1}  with \Cref{fig:BirkhoffExampleColor2}). All circles  outside of the reducible configuration and arcs-not-painted-yellow remain the same  (red, blue, or green). Therefore, the colored circles outside the reducible configuration, like the green circle in \Cref{fig:BirkhoffExampleColor1}, do not play a role in the state-extension process and can be safely ignored (no Kempe switches are needed). 

Removing the circles outside of the configuration reduces the problem to studying only the arcs that exit and then re-enter the configuration region (see \Cref{fig:Basic-cap-and-planar-cap}).  These sets of arcs can be enumerated as follows:
\begin{enumerate}
\item First, there is a finite list of ways the six arcs can exit and re-enter the configuration with a minimum number of interactions between arcs (only what is needed for planarity). Call each such set of arcs and interactions a {\em basic cap}.  
\item Then, for each basic cap, there is a finite list of combinations that show up in planar graphs in which the six arcs in that basic cap can interact with each other in a state graph. Call each such set of arcs and interactions a {\em planar cap}.
\end{enumerate}

\begin{figure}[h]
\includegraphics[scale=0.45]{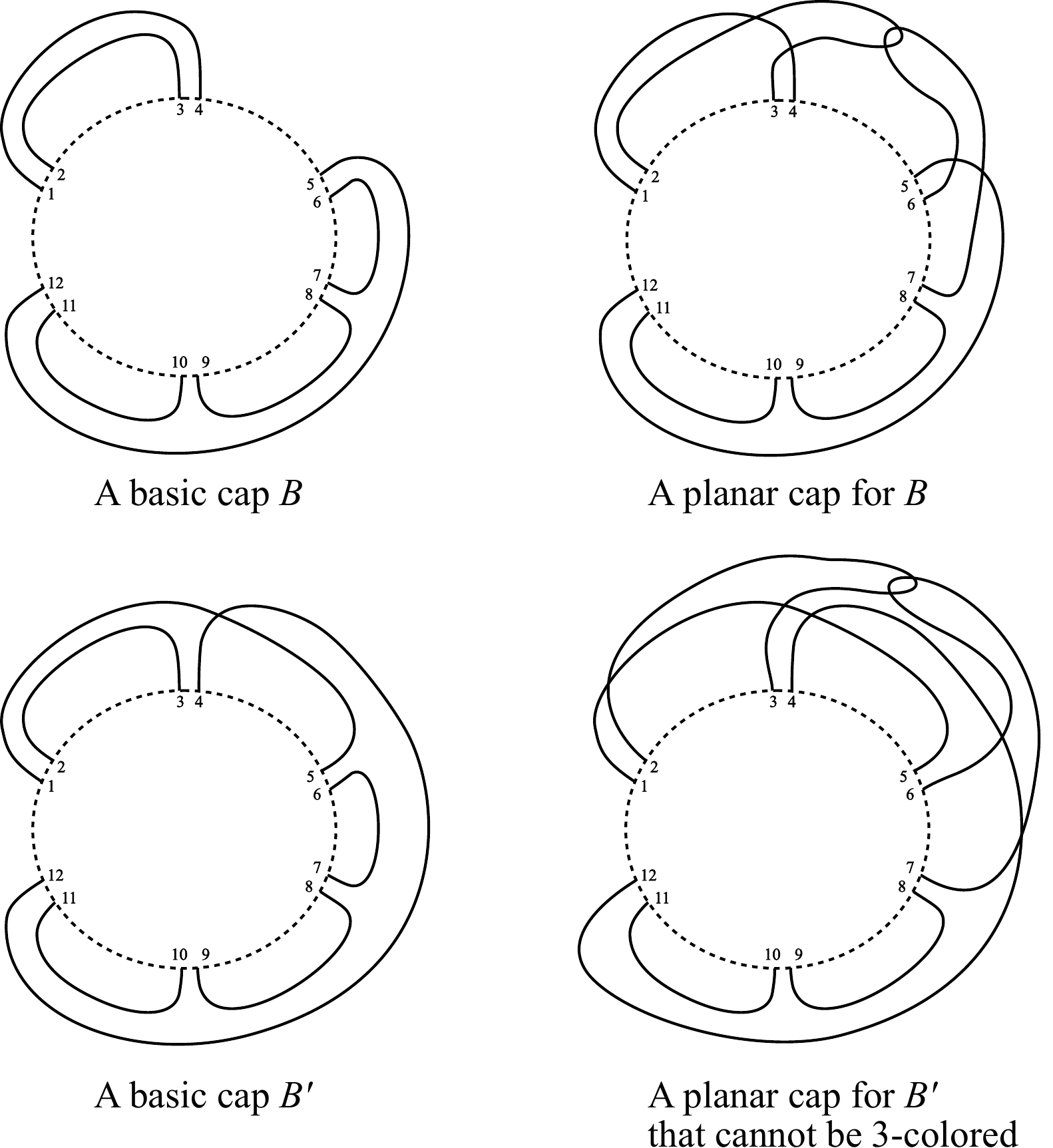}
\caption{Examples of basic caps and  planar caps. The basic cap is minimal in the sense of having only the interactions necessary for planarity: $B$ has no extra interactions between arcs while $B'$ must have $arc[1,5]$ intersect $arc[4,12]$ to be planar. The planar cap of $B$ shows an example of $arc[2,3]$ interacting with $arc[6,7]$. Wherever the arcs intersect or share a spoke, those arcs must be different colors in an $n$-face coloring.}
\label{fig:Basic-cap-and-planar-cap}
\end{figure}

\Cref{fig:Basic-cap-and-planar-cap} shows examples of each. The set of basic caps is clearly finite. For each basic cap there is only a finite number of ways for the six arcs to interact with each other. Therefore, the set of planar caps must also be a finite set.  Denote the entire set of planar caps, which include the basic caps, by $\mathcal{P}_6$. The fact that $\mathcal{P}_6$ is a finite set becomes the basis for proving \Cref{Theorem:Main}.

\bigskip

The arcs in a planar cap may or may not be able to be colored by $n$ colors so that pairs of arcs at a boundary edge of a configuration (called a {\em spoke}) or interacting arcs are painted different colors. The planar cap for $B'$ on the right in \Cref{fig:Basic-cap-and-planar-cap} is an example of a planar cap that cannot be $3$-face colored, but can be $4$-face colored. In this example, if $arc[1,5]$ is red, $arc[2,3]$ must be a different color, say blue, since the two arcs represent different faces adjacent to the edge between $1$ and $2$, called spoke 1, of the configuration (the spoke edge is not shown). In this paper, the number of combinations that need to be checked can always be reduced by a factor of $6$ for $3$-face colorings by fixing the arc exiting at $1$ to be red and the arc exiting at $2$ to be blue.  Since $arc[4,12]$ is next to $arc[2,3]$ at  spoke 2, it cannot be blue. The arc cannot be red because it intersects $arc[1,5]$. Label it green. Note that $arc[6,7]$ is adjacent to $arc[1,5]$ at spoke 3 and intersects $arc[2,3]$ and $arc[4,12]$. Hence it must be colored a fourth color, yellow. The remaining arcs can be given different but valid proper $4$-face colorings. Hence, this example shows that sometimes the caps can be colored by three colors, like basic cap $B$ (try it), and sometimes they cannot. 

Even if a cap can be $3$-colored, inserting a state of a configuration into it may give rise to a set of circles that are not $3$-colorable. For example, \Cref{fig:connect-sum-process} shows how to insert a state $BD_\beta$ of the Birkhoff diamond into the basic cap $B$ from \Cref{fig:Basic-cap-and-planar-cap}. Call the glued together cap and state $B\sharp BD_\beta$ on the right a {\em capped state}.  

\begin{figure}[H]
\includegraphics[scale=0.40]{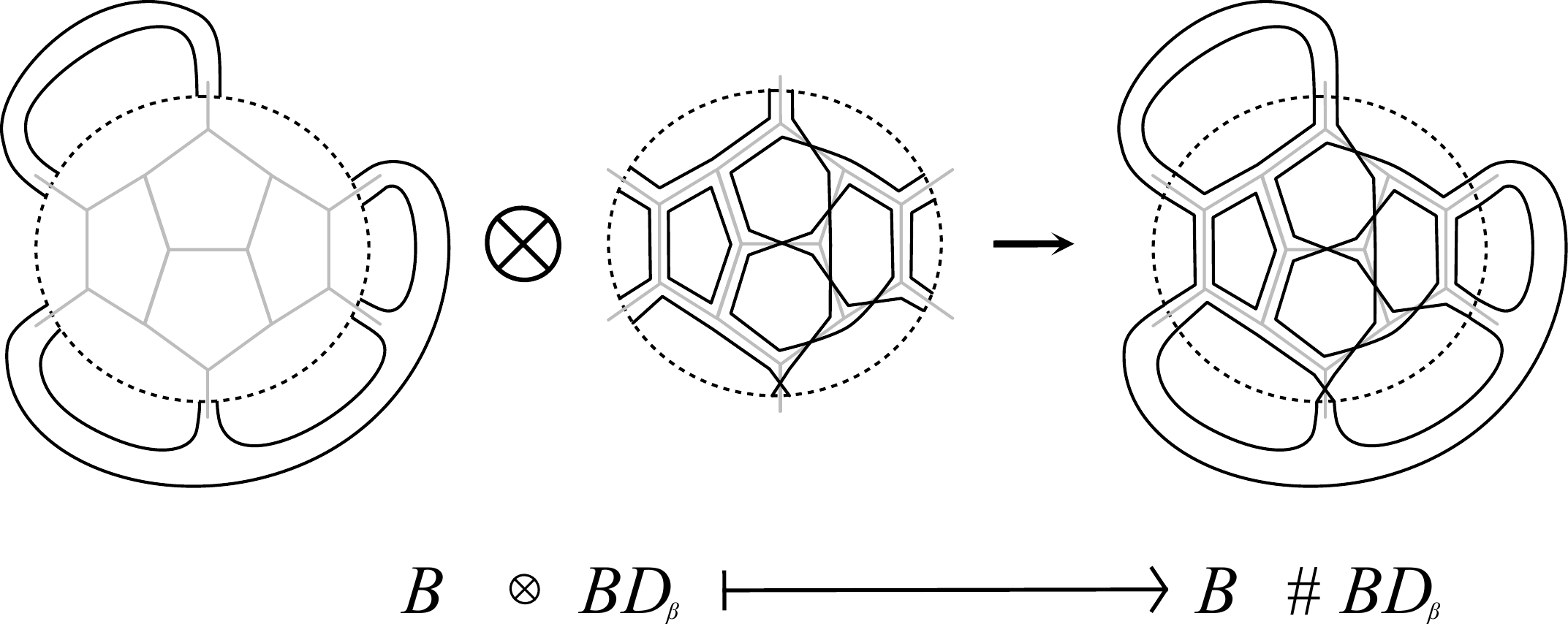}
\caption{The capped state $B\sharp BD_\beta$ is found by inserting the state $BD_\beta$ into the (basic) cap $B$ from \Cref{fig:Basic-cap-and-planar-cap}.}
\label{fig:connect-sum-process}
\end{figure}

In this example, one can check that there is no $3$-face coloring of $B\sharp BD_\beta$, but there is a $4$-face coloring.  Again, a proper coloring here is a choice of colors for each circle of $B\sharp BD_\beta$ such that pairwise circles that interact along an edge in the Birkhoff diamond are different colors (which takes care of the spoke edges as well) or intersecting circles in the cap (which there are none for this example) are different colors. 

Notice that the cap $B$ has a $3$-coloring but the relative state $BD_\beta$ does not. However, there is a capped state $B\sharp R_\alpha$ of a state $R_\alpha$ of the Birkhoff reducer $R$ from \Cref{Theorem:Main} that supports a $3$-face coloring $c'$, and that the $3$-coloring on $B$ can be used to give a $4$-face coloring $c$ to the capped state $B\sharp BD_\beta$ (see \Cref{fig:ThreeColorOnBirkhoffReducer}). This shows that the $3$-face coloring on the capped state $B\sharp R$ is state-extendible to a $4$-face coloring on the capped state $B\sharp BD_\beta$.

\begin{figure}[h]
\includegraphics[scale=0.35]{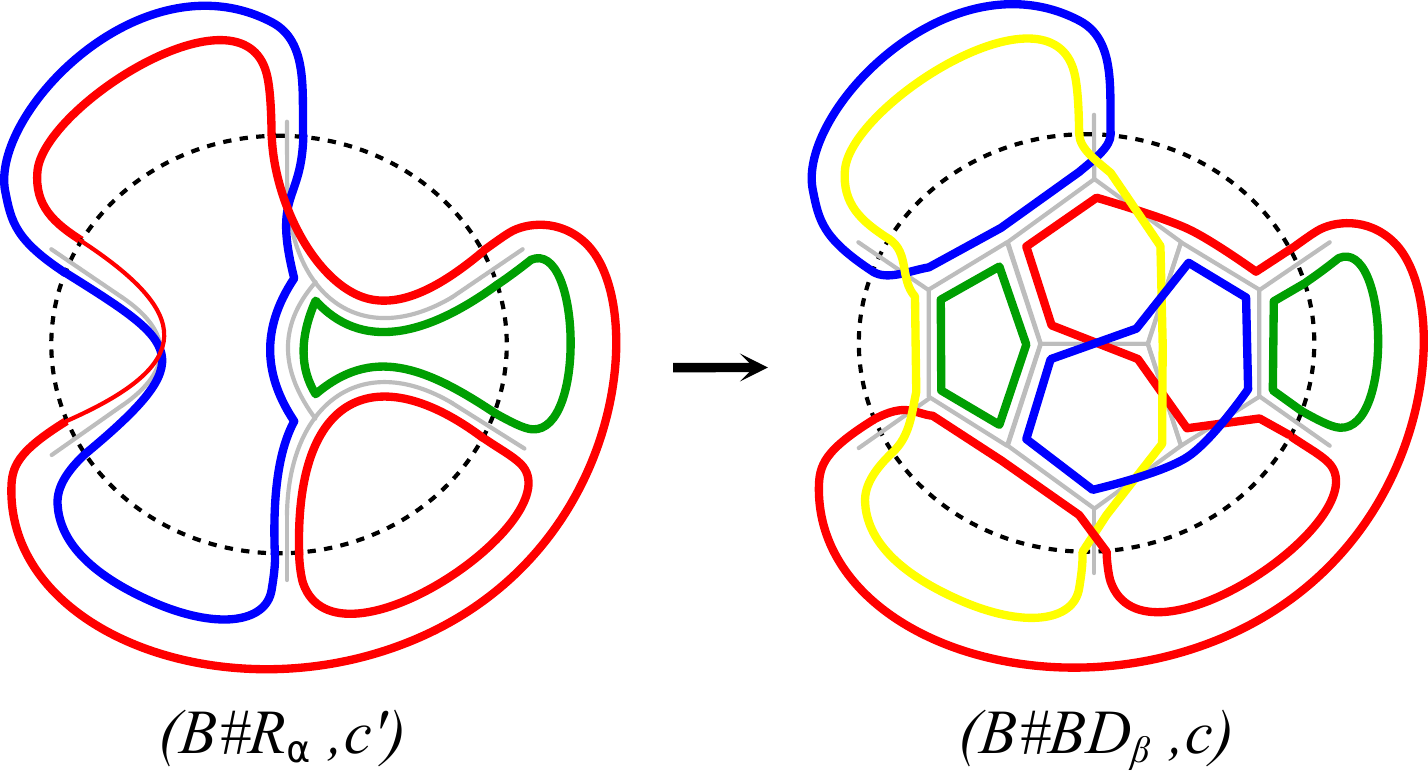}
\caption{A $3$-face coloring on a capped state $B\sharp R_\alpha$ of the reducer that leads to a $4$-face coloring on a capped state $B\sharp BD_\beta$ of the Birkhoff diamond. This shows state-extendibility   on the level of capped states. Compare this picture to \Cref{fig:BirkhoffExampleColor1} and \Cref{fig:BirkhoffExampleColor2}.}
\label{fig:ThreeColorOnBirkhoffReducer}
\end{figure}

Thus, once we prove the next theorem, \Cref{Theorem:Main} reduces to showing that, for all planar caps $P\in \mathcal{P}_6$ and for all relative states $R_\alpha$ of the Birkhoff reducer,  every proper $3$-face coloring on the capped state $P\sharp R_\alpha$ is state-extendible to a $4$-face coloring on some capped state $P\sharp BD_\beta$ of the Birkhoff diamond. Since we are always working with states with colorings, states with a  {\em topological bridge} can be ignored, that is, an edge that shares the same face on both sides of the edge (cf.  columns 1, 2, and 3 of \Cref{fig:ThetaCube}).

\begin{theorem}
Let $\Gamma_\alpha$ be a state graph of a plane graph $\Gamma$ with configuration $C$ with $k$ boundary spokes. Write  $\Gamma_\alpha$ as $D_\beta\sharp C_{\beta'}$ where $D_\beta$ is the relative state outside of the configuration $\Gamma\setminus C$ and $C_{\beta'}$ is the relative state of the configuration.  Then there exists a map $$cap:\{\mbox{state graphs of $\Gamma$ without topological bridges}\} \ra \mathcal{P}_k$$ that takes only the arcs of $D_\beta$ to a unique planar cap $P\in \mathcal{P}_k$.
\label{theorem:cap-map}
\end{theorem}

The rest of this section describes the map $cap$ and shows that it is well-defined. We will work with $\mathcal{P}_6$, which is the case for the Birkhoff diamond, but the same proof works for any $k>1$. 

\begin{example} In \Cref{fig:BirkhoffExampleColor1}, the cap of the state graph of the reducer is $B$ from \Cref{fig:Basic-cap-and-planar-cap}.
\end{example}  

\begin{proof}
Given a state $\Gamma_\alpha$, the circles of $\Gamma_\alpha$ can interact at an edge in two ways: (1) two arcs can be adjacent along an edge but otherwise embedded in the plane (no twist in the band), or (2) two arcs can be adjacent along an edge and intersect each other (there is a half-twist).  See the lefthand side of \Cref{fig:ArcInteractions}.  

\begin{figure}[H]
\includegraphics[scale=0.35]{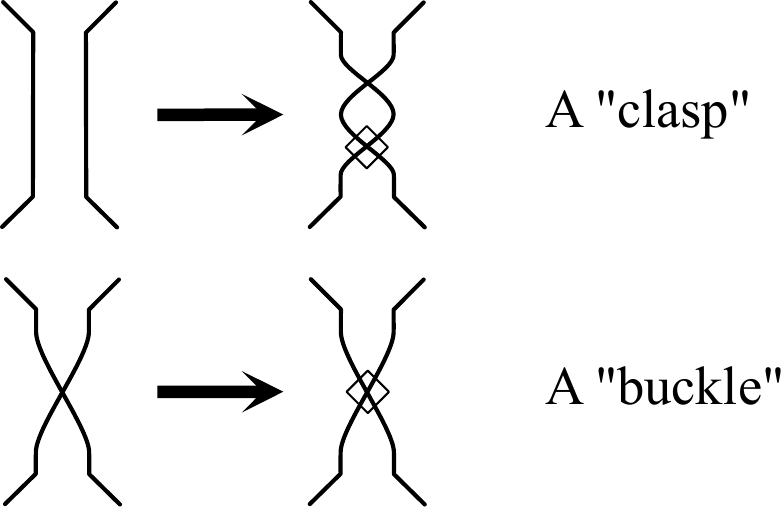}
\caption{To indicate that two circles interact along an edge, insert a diamond symbol as shown to indicate a {\em clasp} or {\em buckle}.}
\label{fig:ArcInteractions}
\end{figure}

When two arcs (and the circles they are part of) interact along an edge as in (1),  place a ``clasp'' in the diagram to indicate the interaction. Similarly, place a ``buckle'' if the two arcs intersect as in (2). We place a diamond symbol on one of the virtual crossings to indicate that the two circles must be painted different colorings when $n$-face coloring the state.  See the righthand side of \Cref{fig:ArcInteractions}. The picture on the left of \Cref{fig:SimplifiedStates} is an example of the state graph in \Cref{fig:BirkhoffExampleColor1} with clasp and buckle interactions. 

\begin{figure}[H]
\includegraphics[scale=0.4]{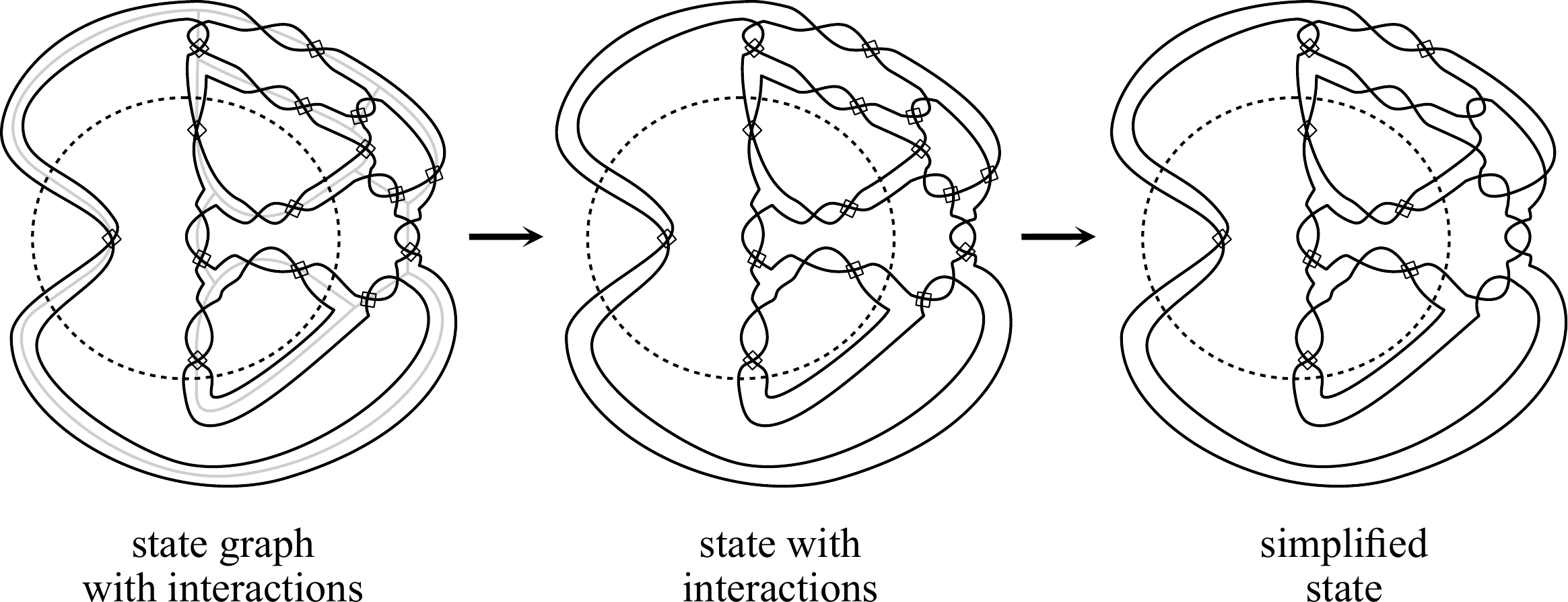}
\caption{A picture of the state graph of the Birkhoff reducer shown in \Cref{fig:BirkhoffExampleColor1} with interactions and its simplification.}
\label{fig:SimplifiedStates}
\end{figure}

With clasps and buckles inserted into the state graph, the underlying  graph structure of the immersed circles is no longer needed for determining proper colorings and can be deleted (see the middle picture of \Cref{fig:SimplifiedStates}).  This means, analogous to what is done in knot theory (cf. \cite{BKR}), that the immersed set of circles in a state can be  simplified. For example, if two circles  have more than two interacting diamond symbols, all but one of the diamond symbols can be removed (see the simplified state picture of \Cref{fig:SimplifiedStates}). The remaining virtual crossings can then be used to simplify the diagram, often drastically (see the cap in \Cref{fig:TakingTheCap} below). If one circle has a self-interacting diamond symbol with itself (a topological bridge), then that state cannot support any proper colorings and is not included.  A version of using the diamond symbol for a special type of interaction was introduced in \cite{Kauffman}, see also \cite{BKM}. It was generalized to multi-virtual crossings in \cite{Kauffman4}.

Removing the extra diamonds and simplifying a state using virtual crossing  transforms the relative state $D_\beta$ outside of the configuration into a set of interacting circles and arcs (see the left picture of \Cref{fig:TakingTheCap}). Note: We can also remove the diamonds on each of the spokes of $D_\beta$ since we have taken the spokes to be part of the state $C_{\beta'}$ of the configuration $C$, so different states of $C$ will capture those interactions.  Recall from the beginning of \Cref{section:taking-the-cap-of-a-state} that the circles (faces) in $D_\beta$ can be ignored for the purpose of state-extending a coloring.  After removing these circles from $D_\beta$, only the arcs and a small number of diamond interactions remain. The new diagram can usually be simplified even more (see the right picture of \Cref{fig:TakingTheCap} after the highlighted circle in $D_\beta$ is removed).  

\begin{figure}[h]
\includegraphics[scale=0.40]{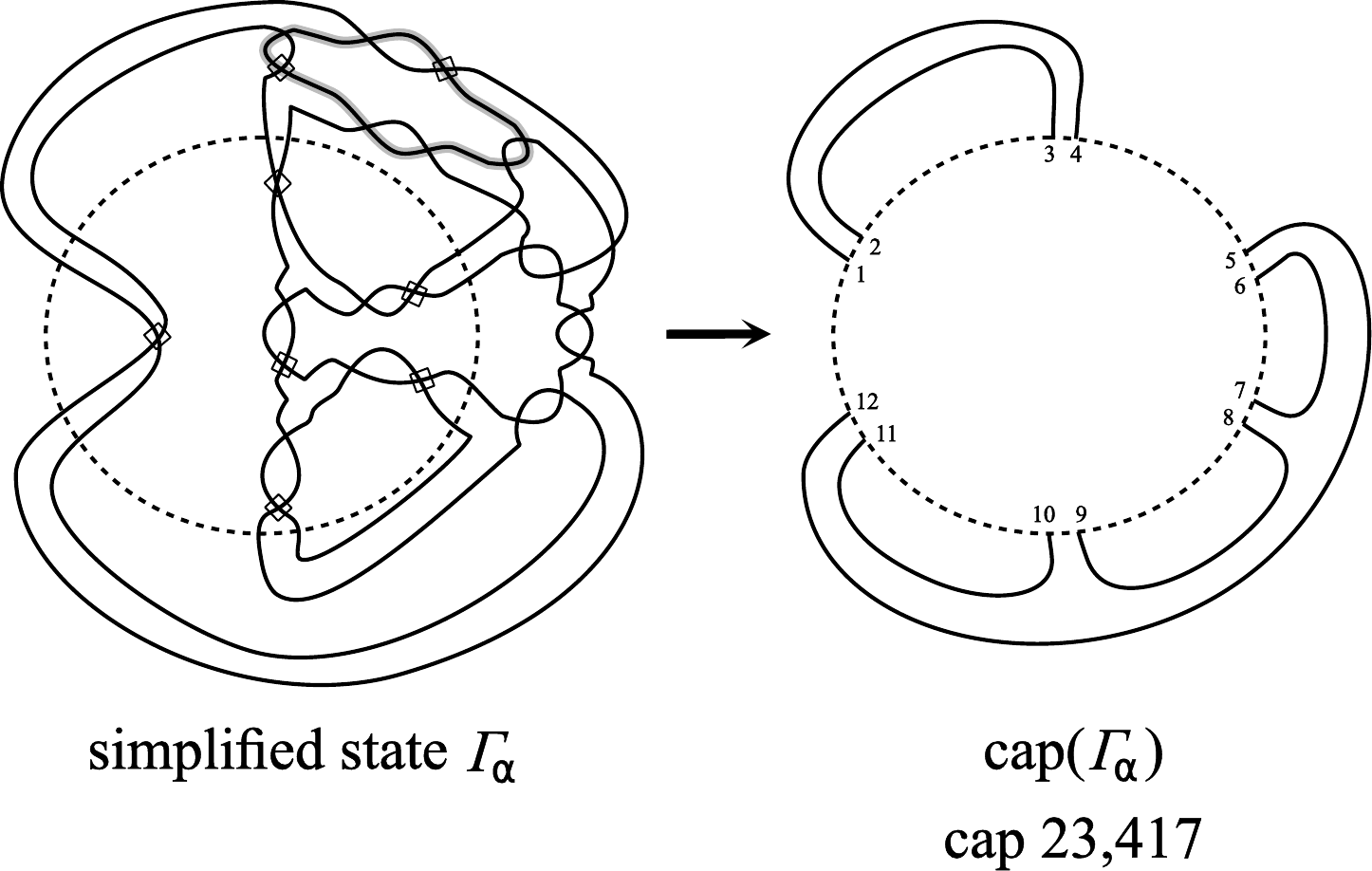}
\caption{The state from \Cref{fig:BirkhoffExampleColor1} and \Cref{fig:SimplifiedStates} simplified to a cap.}
\label{fig:TakingTheCap}
\end{figure}

The result will be a set of immersed arcs that exit and then re-enter the configuration region, along with diamond interactions between different arcs that do not share a spoke. For example, the cap of the simplified state in   \Cref{fig:TakingTheCap} is just the basic cap $B$ from \Cref{fig:Basic-cap-and-planar-cap}, \Cref{fig:connect-sum-process}, and \Cref{fig:ThreeColorOnBirkhoffReducer}.  A slightly more complicated example is a state whose cap looks like the planar cap of the basic cap $B'$ in \Cref{fig:Basic-cap-and-planar-cap}. In this case, the cap with diamond interactions would look like the picture in \Cref{fig:Cap-wth-diamonds}. Notice that we do not need to put diamonds on arcs that intersect if the arcs share a spoke. This helps reduce the number of caps needed in the code described in \Cref{section:main-computation} below.

\begin{figure}[H]
\includegraphics[scale=0.35]{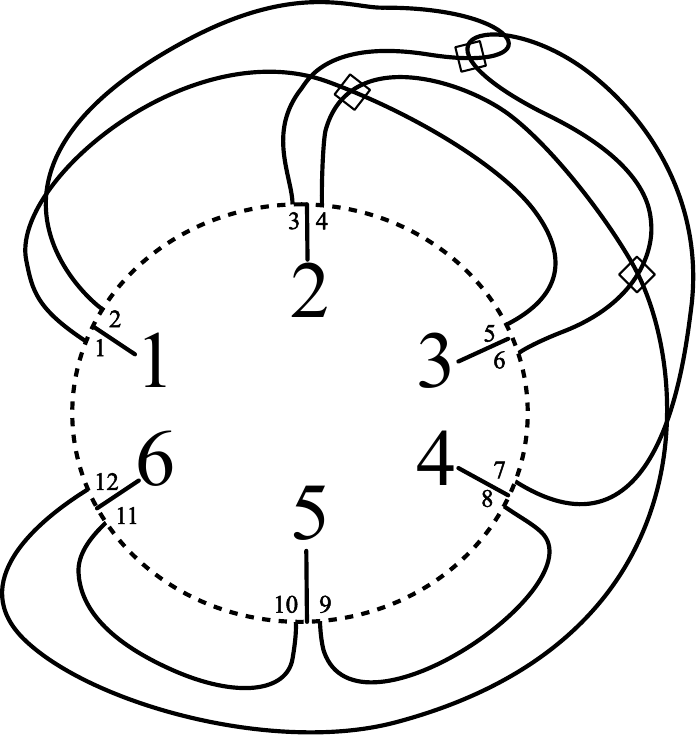}
\caption{What the planar cap of \Cref{fig:Basic-cap-and-planar-cap} looks like with the diamond interactions inserted.}
\label{fig:Cap-wth-diamonds}
\end{figure}

This process always reduces a state to a set of arcs and their interactions (a planar cap), which proves the theorem.
\end{proof}

The data of a cap can be encoded for computation; we use \Cref{fig:Cap-wth-diamonds} to generate an example. First, label the spokes $1,2,\dots,k$ clockwise starting in the upper left part of the configuration region (this is  a choice we started following early in our coding and it quickly became the convention).  Next, label the arc endpoints on either side of the spoke following the same order $1,2,3,4\dots, 2k-1, 2k$. Label the arcs using $arc[a,b]$ where $a$ and $b$ are from this list of labels; generally take $a$ to be the smallest (but not required). Label the diamond interactions by $V[a,c]$ where $a$ is the smallest number in the arc $arc[a,b]$ and $c$ is the smallest number in the arc $arc[c,d]$. Therefore, the diamond interaction between $arc[2,3]$ and $arc[6,7]$ would be labeled $V[2,6]$. With these conventions set, the cap of \Cref{fig:Cap-wth-diamonds}  is represented as the product:
\begin{equation}\label{equation:cap-data}
arc[1,5] \cdot arc[2,3] \cdot arc[4,12] \cdot arc[6,7] \cdot arc[8,9] \cdot arc[10,11] \cdot V[1,4] \cdot V[2,6] \cdot V[4,6].
\end{equation}
Note:  we do not need to label interactions in the cap for arcs already adjacent at a spoke.

The basic cap for this cap can be realized by deleting $V[2,6]$ and $V[4,6]$. However, the $V[1,4]$ must remain for the cap to be planar. The basic cap of a planar cap can always be found by deleting diamond interactions in this manner. 

 \section{The code for proving \Cref{Theorem:Main}} 
\label{section:main-computation}

Given \Cref{theorem:cap-map} and the finiteness of $\mathcal{P}_6$ described in the last section, the proof of \Cref{Theorem:Main} is accomplished by checking, by computer, the following proposition.

\begin{proposition}\label{proposition:3-face-to-4-face-Birkhoff}
Let $BD$ be the Birkhoff diamond  and $R$ be its reducer  (see \Cref{fig:BirkhoffExampleColor0}). For all planar caps $P\in \mathcal{P}_6$ and relative states $R_\alpha$ of the reducer, if $(P\sharp R_\alpha, c')$ is an enhanced capped state with a $3$-face coloring $c'$, then there exists a relative state $BD_\beta$ such that $(P\sharp BD_\beta, c)$ is an enhanced capped state with a $4$-face coloring $c$. The $4$-face coloring $c|_P$ on the planar cap $P$ is identical to the $3$-face coloring $c'|_P$ except that at most two non-interacting arcs have been changed to the fourth color.
\end{proposition}

In this section, we describe the code for the following steps that check \Cref{proposition:3-face-to-4-face-Birkhoff}:
\begin{enumerate}
\item generate all $34,179$ planar caps from $130$ basic caps (\Cref{sec:generating-the-caps}), 
\item insert the $2^6=64$ relative states $R_\alpha$ of the reducer into each planar cap to find $3,744$ planar caps that support a $3$-face coloring  for a total of $7,210$ colored planar caps $(P,c'|_P)$ (\Cref{sec:coloring-the-caps}), and
\item produce, for each colored cap $(P,c'|_P)$, a relative state $BD_\beta$ from the possible $2^{21}=2,097,152$ relative states of the Birkhoff diamond such that $P\sharp BD_\beta$ can be $4$-colored in accordance with the proposition above (\Cref{section:results}). 
\end{enumerate}

Each step above is encoded in a Mathematica notebook starting with ``ProofCodeX,'' where ``X'' is step $1,2,$ or $3$. The three  notebooks can be found as ancillary files to the arXiv version of this paper and will be referred to as the ``first, second, and third notebook.''

\begin{remark} The three steps above for \Cref{proposition:3-face-to-4-face-Birkhoff} can be programmed for the $632$ other configurations and reducers (not just the Birkhoff diamond). Thus, this gives an alternate way to verify the correctness of reducibility in the proof  of the four color theorem.
\end{remark}

\subsection{Generating the planar caps}\label{sec:generating-the-caps}
We will refer to the list of all possible caps for the Birkhoff diamond (and its reducer) as PlanarSixCaps.  The Mathematica file \emph{ProofCode1-6CapGenerator.nb} results in a single output file called \emph{PlanarSixCaps.mx}.  Each cap will consist of 6 arcs, each labeled $arc[a,b]$ as described above, and some number of interactions, each labeled $V[a,b]$.  To generate the list, we will proceed in four basic steps:
\begin{enumerate}
\item Enumerate the ``pre-caps" (described below).
\item For each pre-cap we generate a basic cap with a minimal set of interactions.
\item For each basic cap, we generate a cap for each possible configuration of interactions between the cap arcs.
\item Sift out any non-planar caps that arise from this process to get the set of planar caps.
\end{enumerate}

Initially, we number the six perfect matching edges, the  spokes, at the boundary of the Birkhoff diamond region from one to six.  After this identification, a pre-cap is a fixed-point free permutation of the six perfect matching edges (see \Cref{fig:PreCap}).  The cycles of the permutation determine a set of six arcs that connect the spokes.

\begin{figure}[H]
\includegraphics[scale=0.45]{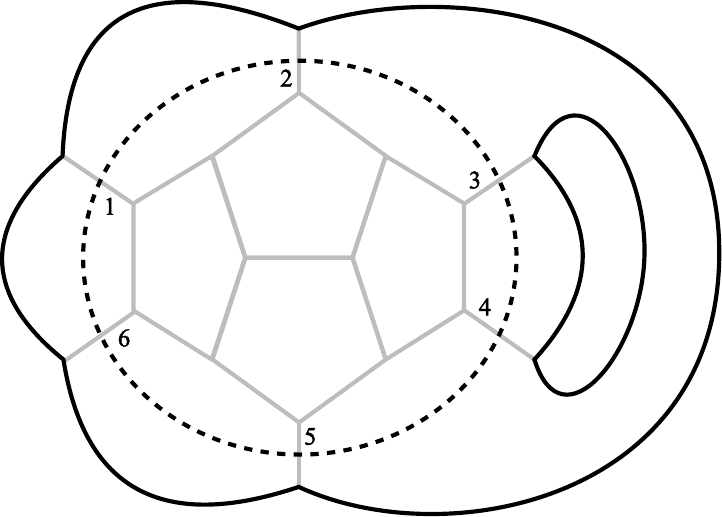}
\caption{A pre-cap described by permutation $\sigma=(1256)(34)$.}
\label{fig:PreCap}
\end{figure}

For the pre-cap shown in \Cref{fig:PreCap}, the corresponding permutation is given by $\sigma = (1 2 5 6) (3 4)$.  By default, Mathematica generates the permutation as a vector:  $\sigma = \{2, 5, 4, 3, 6, 1\}$.  For each permutation, $\sigma$, convert it to a list of ordered pairs, $$\{ \{1,\sigma(1)\},\{2,\sigma(2)\},\{3,\sigma(3)\},\{4,\sigma(4)\},\{5,\sigma(5)\},\{6,\sigma(6)\}\}.$$  Sift out any permutation with a fixed point, which corresponds to an arc that exits and re-enters at the same spoke and therefore cannot be given any proper coloring (a topological bridge).  This is done by searching each permutation for an ordered pair of the form $\{a,a\}$, and if one is found, the permutation is set to zero.  Finally, convert each list of ordered pairs to a product of arcs via the rule $\{a,b\} \mapsto arc [a,b]$.  Duplicates are removed, along with all zeroes, and the list is sorted. 

For the second step, we convert the pre-cap to a basic cap.  Since each arc may enter the Birkhoff diamond region along a clasp or a buckle, there are $2^6$ possible ways to obtain a cap from a pre-cap.  However, only one such cap is needed for each pre-cap since our calculations will consider all possible states of the Birkhoff diamond or its reducer, which include these $2^6$ possible ways.  

The example in \Cref{fig:PreCapToCap} can be used to show how to covert a pre-cap into a basic cap. First, it has two cycles: $$\left(arc[1,2]arc[2,5]arc[5,6]arc[6,1]\right)\cdot \left(arc[3,4]^2\right).$$
The strategy is to alternately double each arc label, or double and subtract one, following the order determined by the cycle.  That is, for a four-cycle we apply the rule:
$$arc[a,b]arc[b,c]arc[c,d]arc[d,a] \mapsto arc[2a,2b-1]arc[2b,2c-1]arc[2c,2d-1]arc[2d,2a-1].$$  Two-cycles and three-cycles are handled similarly.  For the four-cycle above, this results in $arc[2,3]arc[4,9]arc[10,11]arc[12,1]$ and the two-cycle becomes $arc[6,7]arc[8,5]$.   Arc data is treated as orderless, but after sorting, the arc data is stored as 
$$arc[1,12]arc[2,3]arc[4,9]arc[5,8]arc[6,7]arc[10,11].$$  
For each cap in which planarity requires two arcs intersect, for example $arc[a,b]$ and $arc[c,d]$ with $a<c<b<d$, then an arc interaction $V[a,c]$ is inserted to make it planar. At this point,  this process generates 130 basic caps, each with a minimum set of of interactions. The example on the right in \Cref{fig:PreCapToCap} is number 91 in this list.

\begin{figure}[H]
\includegraphics[scale=0.45]{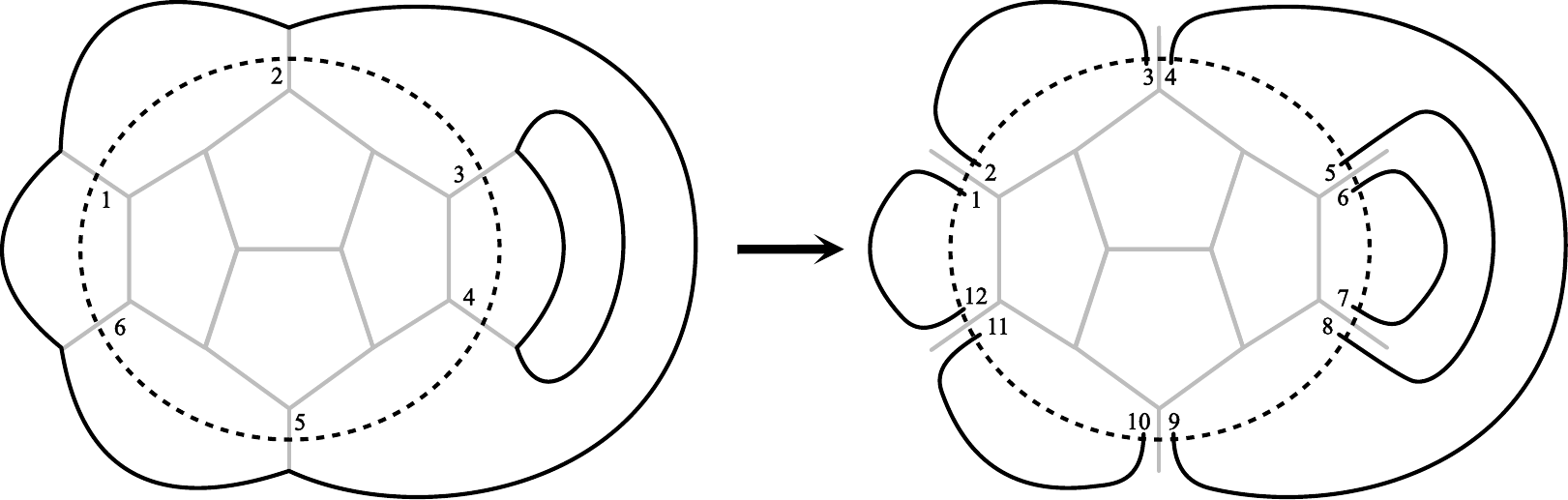}
\caption{ Converting a pre-cap to basic cap 91.}
\label{fig:PreCapToCap}
\end{figure}

For the third step, every possible set of arc interactions is generated, even those that contain an interaction at a spoke or cannot be realized by a planar graph. This is more than needed but simplifies the code and ensures that all possible interactions are attained. First, index the six arcs in a cap by the set $\{1,2,3,4,5,6\}$ and generate the $Comb(6,2)=15$ ordered pairs $\{(1,2), (1,3),\dots, (5,6)\}$. To get all possible interactions, take  all possible subsets of this set where the empty set corresponds to the basic cap.

For each ordered pair, $(i,j)$, in a subset, let $a$ the first label of arc $i$ and let $b$ be the first label of arc $j$ in the basic cap. (Due to how Mathematica sorts, these are always the minimum number in the arc.)  Such an interaction is notated by $V[a,b]$. It represents a buckle or clasp between arcs with an $a$ and $b$ in them (cf. $V[2,6]$ in \Cref{equation:cap-data} of \Cref{fig:Cap-wth-diamonds}).  

For each basic cap generated in step 1, we multiply the basic cap by the product of interactions, one for each of the $2^{15}=32768$ possible subsets of interactions.  Most of these interactions include an interaction at a spoke, which do not need to be considered (as in step 1), and can be zeroed out.  There will also be  duplicate interactions because many basic caps (like basic cap $B'$ in \Cref{fig:Basic-cap-and-planar-cap}) already have interacting arcs. Thus, we can delete zeros and duplicate caps to pare  down the list to $44,707$ caps.

For the fourth step, this list of caps contains all possible planar caps, but it also contains some caps which can only occur in a state for a non-planar graph.  For the cap shown in \Cref{fig:PreCapToCap} we could add an interaction between $arc[1,12]$ and $arc[5,8]$ by multiplying basic cap 91 by $V[1,5]$ to obtain:
\begin{equation}
arc[1,12]arc[2,3]arc[4,9]arc[5,8]arc[6,7]arc[10,11]V[1,5].
\label{eqn:nonPlanarCap}
\end{equation}
However, because $arc[4,9]$ separates $arc[1,12]$ from $arc[5,8]$ this cap cannot occur in any state for a planar graph.  The reason is that for a planar graph every crossing in the cap has to come from either a buckle or a clasp.  To add the interaction $V[1,5]$ one must also multiply by $V[1,4]$ or $V[4,5]$.  After doing so, the resulting cap may occur in a state for a planar graph (see \Cref{fig:TwoPlanarCaps}).  Our final step is to zero out and remove any caps where neither of these necessary additional interactions are present.

\begin{figure}[H]
\includegraphics[scale=0.45]{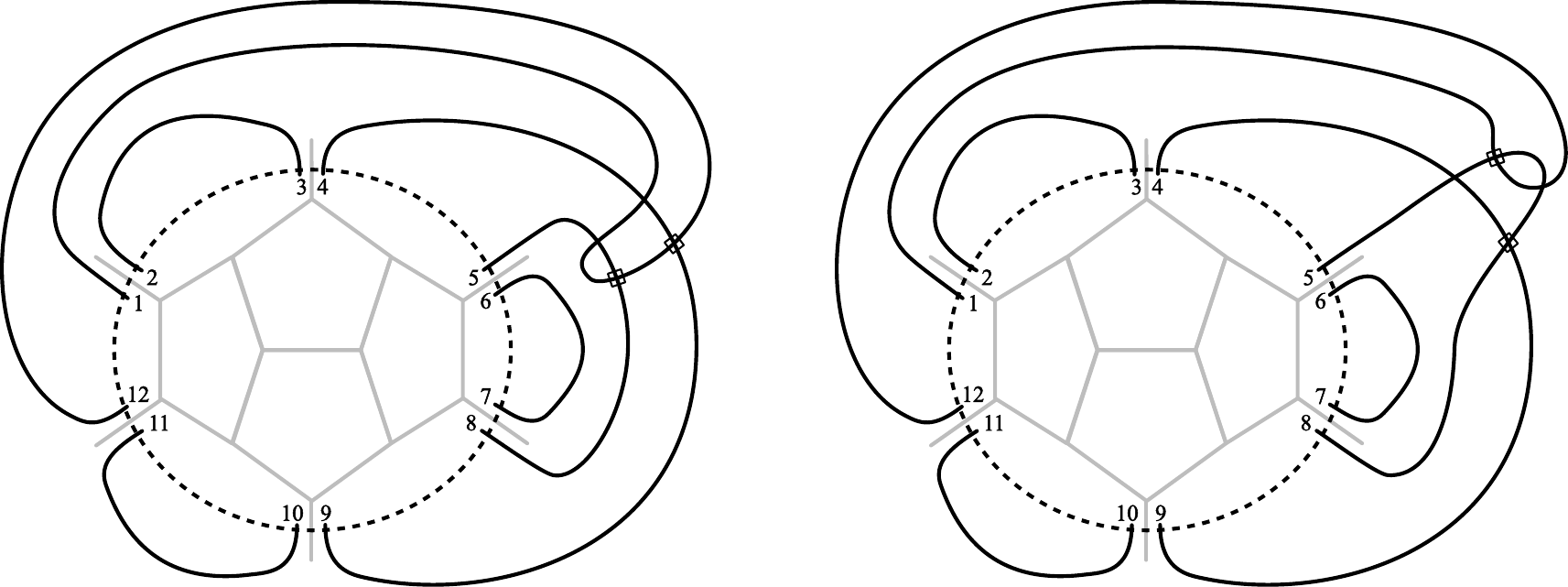}
\caption{Two planar caps for $V[1,4]$ and $V[4,5]$ for basic cap number 91.}
\label{fig:TwoPlanarCaps}
\end{figure}

The  result obtained by this Mathematica notebook is a list of every possible planar cap that can occur as the cap of a state graph without a topological bridge, outputted as a compressed dataset  in \emph{PlanarSixCaps.mx}. This file is included in the ancillary files and can be imported with an {\em Import} command.  This process results in a total of 34179 caps.  The two caps shown in \Cref{fig:TwoPlanarCaps} are caps $18835$ and $18840$, respectively.

\subsection{Coloring the caps}\label{sec:coloring-the-caps}
The next task, which is the output of the second notebook \emph{ProofCode2-ColoredCapsOnReducer.nb},  is to produce a list of all possible $3$-face colorings of the planar caps that are compatible with a $3$-face coloring on the Birkhoff reducer. To do this, we modify the arc notation used in the previous section to generate the planar six caps to include coloring information.  This modification is given by
$$arc[a,b] \mapsto arc[c,A[a,b]].$$
The color is indicated by $c$.  Initially, $c$ is given a value of $0$, which is thought of as an uncolored arc.  When assigning one of three colors to each arc, we will use a value of $1, 2$ or $3$, which will be interpreted as red, blue or green, respectively.  The original arc labels are encoded by $A[a,b]$.  

To reduce the number of colored caps, we make the assumption that the two arcs of the form $arc[c_1,A[1,a]]$ and $arc[c_2,A[2,b]]$ are always colored red ($c_1=1$) and blue ($c_2=2$).  For each planar six cap, we color the remaining four arcs in all possible ways and sift out any colorings that are incompatible with the Birkhoff reducer.  The second notebook has detailed explanation of the steps in this process. 

While there are $34179$ planar caps, only $14602$ of them have a $3$-face coloring on the cap (cf. \Cref{fig:ThreeColorOnBirkhoffReducer} and the text above it).  Of these, $3744$ have at least one coloring that can be extended to a $3$-face coloring on some relative state of the reducer (see the lefthand picture of \Cref{fig:ThreeColorOnBirkhoffReducer}, for example). However, for each of these caps, there may be multiple colorings that extend, even after fixing the first and second arc to be red and blue. The total number of $3$-face colorings that extend to a relative state on the reducer is $7210$. This resulting list is exported as \emph{ReducerColoredCaps.mx}.

\subsection{The Results}\label{section:results}
The third and final  notebook is titled \emph{ProofCode3-ColoringTable.nb}.  For each of the colored caps in the file \emph{ReducerColoredCaps.mx}, a $4$-face coloring of some state of the Birkhoff diamond is exhibited, which completes the proof of \Cref{proposition:3-face-to-4-face-Birkhoff}. 

The process by which we obtain the $4$-face colorings is like how we extend the colorings through the Birkhoff reducer in the second notebook, except this time we allow for the ability to change one or two arcs (and circles or arcs of the relative state of the Birkhoff configuration) into a fourth color.  Since this last part of the process is an existence result, we do not provide the code that finds each $4$-face coloring. However, we do provide the list and a way for the reader to check that each is a valid $4$-face coloring.

To present the results in table format, the colorings are imported from a compressed dataset called \emph{BirkhoffResults.mx} along with the colored caps from the file \emph{ReducerColoredCaps.mx}.  The data is then combined and outputted to a table. See \Cref{appendix:table-of-results} for a description of how to read the table and build a picture of any capped state from an entry in the table. There are also two rules, \emph{checkColor} and \emph{checkCapMatch}, available in the third notebook. The first that checks the validity of a coloring of a capped state of the Birkhoff diamond. The second checks that the cap and $3$-face coloring of the reducer corresponds to the cap and $4$-face coloring of the capped state of the Birkhoff diamond. Running these functions on the data from  \emph{BirkhoffResults.mx} (as shown in the code) verifies the proof of \Cref{proposition:3-face-to-4-face-Birkhoff}.

\section*{Acknowledgements} The three Mathematica files included as ancillary files in the arXiv version of this paper run on a reasonably fast laptop in under one hour.  However, this is a vastly simplified, compact version of our original code.  It has been rewritten many times using  different Mathematica functions to improve the efficiency and runtime.\footnote{The fact that the code has been rewritten many times is an important point: We are confident this code is free from computer language errors (due to the way, say, Mathematica may behave unexpectedly in implementing one of its core functions) because each rewrite of the code using different techniques and different Mathematica functions produced the {\em exact} same results.}  The first few iterations of this code often required a supercomputer to run in a reasonable amount of time (days instead of  months).  Plus, we had to test many configurations to discover the Birkhoff reducer in \Cref{Theorem:Main}. This  required having access to a supercomputer.

We are grateful to the Louisiana State University Mathematics Department for providing unfettered access to its own high-end computational system. LSU has much larger, university-wide, supercomputers, but access to them is strictly regulated. It was only by having the freedom to quickly test and experiment with code and examples on a fast system that made this paper possible. For example, at one point during our testing, we used seven of its nodes with 300 CPU cores and about four terabytes of its total memory running almost continuously over a five day period. We would also like to thank Alex Perlis and Nikkos Svoboda of the LSU mathematics department for their expertise in running and maintaining the system and for providing advice on how to improve our code.

\appendix
\section{Conventions used in the Table of Results}\label{appendix:table-of-results}

A small sample of the table of results created by the notebook \emph{ProofCode3-ColoringTable.nb} is given below in \Cref{tab:results}.  The table has four columns:  Column $1$ is the number of the colored cap (from $1$ to $3744$).  Column $2$ specifies the coloring of the cap.  Column $3$ gives the arc data for the colored cap.  Finally, column $4$ specifies a $4$-face coloring of the capped state of the Birkhoff diamond.  The notation for this coloring requires  further explanation.

Consider the $4$-face coloring of the capped state for $(47,3)$ (the last row shown in the table).  The first arc is given by $arc[1, A[4, 9, 16, 21, 28, 38, 40]]$.  The first number, $1$, indicates red, just as it does in the colored caps. By convention, in addition to labeling the sides of each spoke with a number $1$ through $12$, we label each edge of the blowup cycles with a number as shown in \Cref{fig:ArcNotationCircles}.  The list of numbers $A[4, 9, 16, 21, 28, 38, 40]$ represents a collection of arcs that, when joined and combined with the arc in the cap, make up a circle (see \Cref{fig:ArcNotationCircles}). 

\begin{figure}[H]
\includegraphics[scale=0.6]{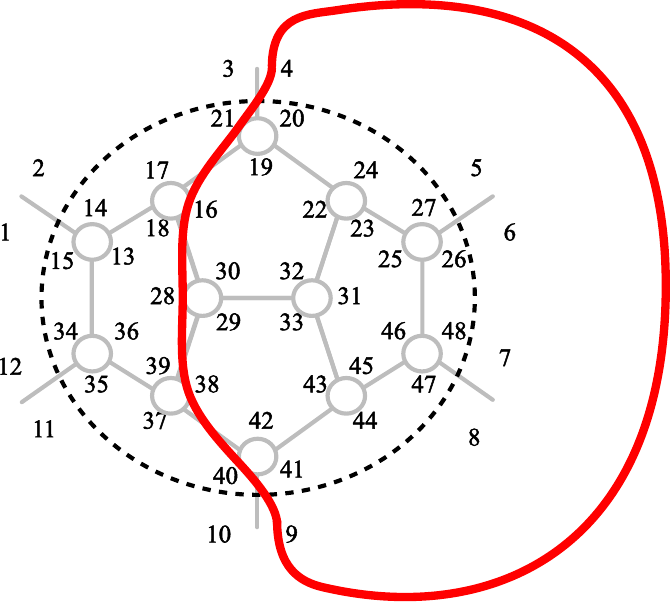}
\caption{Building a circle from arc data.}
\label{fig:ArcNotationCircles}
\end{figure}

\begin{figure}[H]
\includegraphics[scale=0.45]{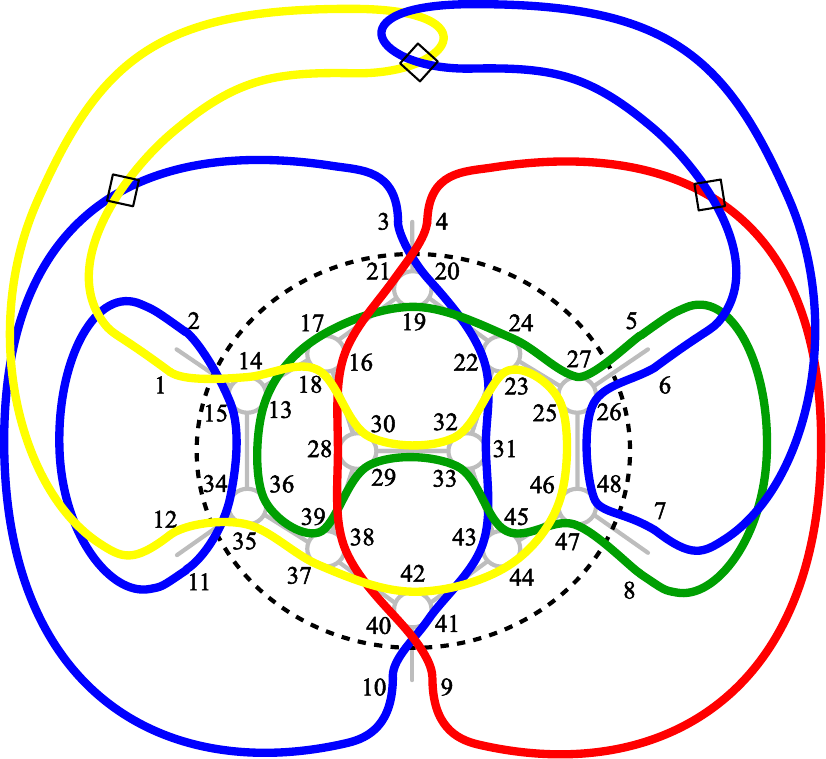}
\caption{Producing a $4$-face colored state from a colored cap.}
\label{fig:ColoredCap2State}
\end{figure}
\mbox{}\\

For the colored cap $(47,3)$ of the reducer, \Cref{fig:ColoredCap2State} shows a $4$-face coloring of a capped state of the Birkhoff diamond.  The colored cap can be seen in the colored arcs outside the dotted region, except that the yellow arc was originally red in the cap (i.e. $arc[1, A[1, 12]]$). 

\begin{table}[H] \small
\begin{center}
\caption{Format of Coloring Result Table}
\label{tab:results}
\begin{tabular}{|c|c|p{55mm}|p{83mm}|}
\hline
Cap & No. & Colored Cap &  $4$-Face Colored State\\
\hline
$1$ & $1$ & arc[1, A[1, 12]] arc[1, A[3, 10]] arc[1, A[5, 8]] arc[2, 
  A[2, 11]] arc[2, A[4, 9]] arc[3, A[6, 7]] &  arc[1, A[5, 8, 26, 48]] arc[1, A[3, 10, 20, 22, 31, 41, 43]] arc[2, 
  A[2, 11, 15, 34]] arc[2, A[4, 9, 16, 21, 28, 38, 40]] arc[3, 
  A[6, 7, 13, 17, 19, 24, 27, 29, 33, 36, 39, 45, 47]] arc[4, 
  A[1, 12, 14, 18, 23, 25, 30, 32, 35, 37, 42, 44, 46]] V[1, 2] V[1, 
  3] V[1, 4] V[1, 5] V[1, 6] V[2, 6] V[3, 4] V[3, 6] V[4, 6] V[5, 6]\\
 
 \hline
 
1 & 2 & arc[1, A[1, 12]] arc[1, A[3, 10]] arc[2, A[2, 11]] arc[2, 
  A[4, 9]] arc[2, A[5, 8]] arc[3, A[6, 7]] &  arc[1, A[3, 10, 20, 22, 31, 41, 43]] arc[2, A[2, 11, 15, 34]] arc[2, 
  A[5, 8, 26, 48]] arc[2, A[4, 9, 16, 21, 28, 38, 40]] arc[3, 
  A[6, 7, 13, 17, 19, 24, 27, 29, 33, 36, 39, 45, 47]] arc[4, 
  A[1, 12, 14, 18, 23, 25, 30, 32, 35, 37, 42, 44, 46]] V[1, 2] V[1, 
  3] V[1, 4] V[1, 5] V[1, 6] V[2, 6] V[3, 4] V[3, 6] V[4, 6] V[5, 6]\\
  
 \hline

\vdots &\vdots & \vdots & \vdots\\

\hline

47 & 2 & arc[1, A[1, 12]] arc[1, A[4, 9]] arc[2, A[2, 11]] arc[2,  A[3, 10]] arc[2, A[5, 8]] arc[3, A[6, 7]] V[1, 3] V[1, 6] V[4, 6] & arc[1, A[4, 9, 16, 21, 28, 38, 40]] arc[2, A[2, 11, 15, 34]] arc[2, A[5, 8, 26, 48]] arc[2, A[3, 10, 20, 22, 31, 41, 43]] arc[3, A[6, 7, 13, 17, 19, 24, 27, 29, 33, 36, 39, 45, 47]] arc[4,  A[1, 12, 14, 18, 23, 25, 30, 32, 35, 37, 42, 44, 46]] V[1, 2] V[1, 3] V[1, 4] V[1, 5] V[1, 6] V[2, 6] V[3, 4] V[3, 6] V[4, 6] V[5, 6]\\
 
\hline

47 & 3 & arc[1, A[1, 12]] arc[1, A[4, 9]] arc[2, A[2, 11]] arc[2, 
  A[3, 10]] arc[2, A[6, 7]] arc[3, A[5, 8]] V[1, 3] V[1, 6] V[4, 6] & arc[1, A[4, 9, 16, 21, 28, 38, 40]] arc[2, A[2, 11, 15, 34]] arc[2, 
  A[6, 7, 26, 48]] arc[2, A[3, 10, 20, 22, 31, 41, 43]] arc[3, 
  A[5, 8, 13, 17, 19, 24, 27, 29, 33, 36, 39, 45, 47]] arc[4, 
  A[1, 12, 14, 18, 23, 25, 30, 32, 35, 37, 42, 44, 46]] V[1, 2] V[1, 
  3] V[1, 4] V[1, 5] V[1, 6] V[2, 5] V[3, 4] V[3, 5] V[4, 5] V[4, 
  6] V[5, 6]\\
 
\hline

\vdots & \vdots & \vdots & \vdots \\

\hline

\end{tabular}
\end{center}
\end{table}

\end{document}